\documentclass{amsart}
\usepackage{amsmath}
\usepackage{amssymb}
\usepackage{amsthm}
\usepackage[msc-links,abbrev]{amsrefs}
\usepackage{hyperref}
\usepackage[noabbrev,capitalize]{cleveref}
\usepackage{mathrsfs}
\usepackage{mathtools}
\usepackage{slashed}

\DeclareMathOperator{\im}{im}

\DeclareMathOperator{\dvol}{dvol}

\DeclareMathOperator{\Imaginary}{Im}
\DeclareMathOperator{\Real}{Re}

\DeclareMathOperator{\rwedge}{\diamond}
\DeclareMathOperator{\krwedge}{\bullet}
\DeclareMathOperator{\hodge}{\star}

\newcommand{\defn}[1]{{\boldmath\bfseries#1}}

\newcommand{\Alpha}{A}
\newcommand{\Beta}{B}

\newcommand{\Lscal}{L_{0,0}^{\mathrm{scal}}}
\newcommand{\Qscal}{Q_{0,0}^{\mathrm{scal}}}

\newcommand{\HBC}{H_{\mathrm{BC}}}
\newcommand{\HA}{H_{\mathrm{A}}}
\newcommand{\HQ}{\mathcal{H}_{\mathrm{Q},\dbbar}}
\newcommand{\HB}{\mathcal{H}_{\mathrm{B},\dbbar}}
\newcommand{\HQdual}{\mathcal{H}^{\mathrm{Q},\dbbar}}
\newcommand{\HBdual}{\mathcal{H}^{\mathrm{B},\dbbar}}
\newcommand{\mHQ}{\mathcal{H}_{\mathrm{Q},d}}

\newcommand{\COmega}{\Omega_{\mathbb{C}}}

\newcommand{\oW}{\overline{W}}

\newcommand{\opsi}{\overline{\psi}}

\newcommand{\oxi}{\overline{\xi}}
\newcommand{\otau}{\overline{\tau}}
\newcommand{\oomega}{\overline{\omega}}

\newcommand{\omD}{\overline{\mathcal{D}}}

\newcommand{\oBox}{\overline{\Box}}

\newcommand{\hA}{\widehat{A}}

\newcommand{\hP}{\widehat{P}}

\newcommand{\hnabla}{\widehat{\nabla}}

\newcommand{\htheta}{\widehat{\theta}}

\newcommand{\lp}{\langle}
\newcommand{\rp}{\rangle}
\newcommand{\lv}{\lvert}
\newcommand{\rv}{\rvert}

\newcommand{\llp}{\lp\!\lp}
\newcommand{\rrp}{\rp\!\rp}

\newcommand{\db}{\partial_b}
\newcommand{\dbbar}{\overline{\partial}_b}
\newcommand{\dhor}{\partial_0}

\newcommand{\mD}{\mathcal{D}}

\newcommand{\mH}{\mathcal{H}}

\newcommand{\mT}{\mathcal{T}}

\newcommand{\kI}{\mathfrak{I}}
\newcommand{\kJ}{\mathfrak{J}}

\newcommand{\kL}{\mathfrak{L}}

\newcommand{\kQ}{\mathfrak{Q}}

\newcommand{\bC}{\mathbb{C}}
\newcommand{\bD}{\mathbb{D}}

\newcommand{\bH}{\mathbb{H}}

\newcommand{\bR}{\mathbb{R}}

\newcommand{\sP}{\mathscr{P}}

\newcommand{\sR}{\mathscr{R}}
\newcommand{\sS}{\mathscr{S}}

\def\sideremark#1{\ifvmode\leavevmode\fi\vadjust{\vbox to0pt{\vss
 \hbox to 0pt{\hskip\hsize\hskip1em
 \vbox{\hsize3cm\tiny\raggedright\pretolerance10000
 \noindent #1\hfill}\hss}\vbox to8pt{\vfil}\vss}}}

\newcommand{\suchthatcolon}{\mathrel{}:\mathrel{}}

\newtheorem{theorem}{Theorem}[section]
\newtheorem{proposition}[theorem]{Proposition}
\newtheorem{lemma}[theorem]{Lemma}
\newtheorem{corollary}[theorem]{Corollary}

\theoremstyle{definition}
\newtheorem{definition}[theorem]{Definition}

\newtheorem{conjecture}[theorem]{Conjecture}

\theoremstyle{remark}
\newtheorem{remark}[theorem]{Remark}

\numberwithin{equation}{section}

\usepackage{tikz-cd}

\begin{document}

\title[$Q$-curvature operators in dimension five]{Some $Q$-curvature operators on five-dimensional pseudohermitian manifolds}
\author{Jeffrey S. Case}
\address{109 McAllister Building \\ Penn State University \\ University Park, PA 16802 \\ USA}
\email{jscase@psu.edu}
%\date{\today}
\keywords{$Q$-curvature; $Q$-flat contact form; invariant differential operators}
\subjclass[2020]{Primary 32V05; Secondary 58A10; 58J10}
\begin{abstract}
 We construct $Q$-curvature operators on $d$-closed $(1,1)$-forms and on $\dbbar$-closed $(0,1)$-forms on five-dimensional pseudohermitian manifolds.
 These closely related operators give rise to a new formula for the scalar $Q$-curvature.
 As applications, we give a cohomological characterization of CR five-manifolds which admit a $Q$-flat contact form; and we show that every closed, strictly pseudoconvex CR five-manifold with trivial first real Chern class admits a $Q$-flat contact form provided the $Q$-curvature operator on $\dbbar$-closed $(0,1)$-forms is nonnegative.
\end{abstract}
\maketitle

\section{Introduction}
\label{sec:intro}

In conformal geometry, the \defn{$Q$-curvature operators} on an even-dimensional Riemannian manifold $(M^{2n},g)$ --- introduced by Branson and Gover~\cite{BransonGover2005} using the tractor calculus and reconstructed by Aubry and Guillarmou~\cite{AubryGuillarmou2011} via Poincar\'e--Einstein metrics --- are a family $Q_k^g \colon \Omega^kM \to \Omega^kM$, $k\in\{0,\dotsc,n\}$, of formally self-adjoint differential operators with the property that
\begin{equation*}
 e^{2(n-k)\Upsilon}Q_k^{e^{2\Upsilon}g}\omega = Q_k^g\omega + c_{k} (d^\ast)^g Q_{k+1}^gd(\Upsilon \omega)
\end{equation*}
for all $\omega \in \ker d \cap \Omega^kM$ and all $\Upsilon \in C^\infty(M)$.
It follows that the \defn{GJMS operator} $L_k^g := c_k (d^\ast)^g Q_{k+1}^gd \colon \Omega^kM \to \Omega^kM$ on forms is conformally covariant.
These generalize Branson's $Q$-curvature~\cite{Branson1995} and the critical-order scalar GJMS operator~\cite{GJMS1992} to differential forms.
On closed Riemannian manifolds, the $Q$-curvature operator determines a conformally invariant pairing
\begin{equation*}
 (\omega, \tau) \mapsto \int_M \lp \omega, Q_k^g\tau\rp_g \, \dvol_g
\end{equation*}
on closed differential $k$-forms, and the conformally invariant vector space
\begin{equation*}
 \mH^k := \left\{ \omega \in \Omega^kM \suchthatcolon d\omega = 0, d^\ast Q_k\omega = 0 \right\}
\end{equation*}
is finite-dimensional.
The canonical morphism $\mH^k \to H^k(M;\bR)$ motivates the description of $\mH^k$ as the space of conformal harmonic $k$-forms.

In this article we explicitly construct two $Q$-curvature operators on pseudohermitian five-manifolds.
The first is a $Q$-curvature operator on $d$-closed $(1,1)$-forms;
the second is a $Q$-curvature operator on $\dbbar$-closed $(0,1)$-forms.
We have chosen to focus on these two operators because they are closely related to each other and to the scalar $Q$-curvature~\cite{FeffermanHirachi2003}.
We have chosen to focus on dimension five because, unlike in higher dimensions, it is a manageable task to produce explicit formulas and, unlike in dimension three, the resulting operators are new.

In conformal geometry, there is only one conformally invariant complex of differential forms, namely the de Rham complex.
As pointed out by Branson and Gover~\cite{BransonGover2005}, the GJMS operator $L_k$ gives rise to a detour complex
\begin{equation*}
 \label{eqn:conformal-detour-complex}
 \begin{tikzpicture}[baseline=(current bounding box.center),yscale=0.8,xscale=1.0]
  \node (first-dots) at (0,0) {$\dotsm$};
  \node (k) at (2,0)  {$\Omega^{k}M$};
  \node (second-dots) at (4,0) {$\dotsm$};
  \node (2n-k) at (6,0)  {$\Omega^{2n-k}M$};
  \node (third-dots) at (8,0) {$\dotsm$};
  % Complex
  \draw [->] (first-dots) to node[above] {\tiny $d$} (k);
  \draw [->] (k) to node[above] {\tiny $d$} (second-dots);
  \draw [->] (second-dots) to node[above] {\tiny $d$} (2n-k);
  \draw [->] (2n-k) to node[above] {\tiny $d$} (third-dots);
  % Detour
  \draw [->] (k.south) -- (2,-1) to node[above] {\tiny $\hodge L_k$} (6,-1) -- (2n-k);
 \end{tikzpicture}
\end{equation*}
Here the horizontal complex is the de Rham complex and the complex obtained by taking the ``detour'' along $\hodge L_k$ is the aforementioned detour complex.

In CR geometry, there are many CR invariant complexes of differential forms, among them the Rumin complex~\cite{Rumin1990}, the Kohn--Rossi complexes~\cites{KohnRossi1965,GarfieldLee1998}, and the CR pluriharmonic complex~\cite{Case2021rumin}.
These complexes are all encoded in the \defn{bigraded Rumin complex}~\cites{Case2021rumin,GarfieldLee1998,Garfield2001}
\begin{equation}
 \label{eqn:bigraded-rumin-complex}
 \begin{tikzpicture}[baseline=(current bounding box.center),yscale=0.8,xscale=1.0]
  \node (0_0) at (0,0)  {$\sR^{0,0}$};
  \node (1_0) at (1,1)  {$\sR^{1,0}$};
  \node (0_1) at (1,-1) {$\sR^{0,1}$};
  \node (2_0) at (2,2)  {$\sR^{2,0}$};
  \node (1_1) at (2,0)  {$\sR^{1,1}$};
  \node (0_2) at (2,-2) {$\sR^{0,2}$};
  \node (3_0) at (4,2)  {$\sR^{3,0}$};
  \node (2_1) at (4,0)  {$\sR^{2,1}$};
  \node (1_2) at (4,-2) {$\sR^{1,2}$};
  \node (3_1) at (5,1)  {$\sR^{3,1}$};
  \node (2_2) at (5,-1) {$\sR^{2,2}$};
  \node (3_2) at (6,0)  {$\sR^{3,2}$.};
  % 0-th db
  \draw [->] (0_0) to node[pos=0.8, left] {\tiny $\db$} (1_0);
  \draw [->] (1_0) to node[pos=0.8, left] {\tiny $\db$} (2_0);
  % 1-st db
  \draw [->] (0_1) to node[pos=0.8, left] {\tiny $\db$} (1_1);
  \draw [->] (1_1) to node[pos=0.3, left] {\tiny $\db$} (3_0);
  % 2-nd db
  \draw [->] (0_2) to node[pos=0.3, left] {\tiny $\db$} (2_1);
  \draw [->] (2_1) to node[pos=0.8, left] {\tiny $\db$} (3_1);
  % 3-rd db
  \draw [->] (1_2) to node[pos=0.8, left] {\tiny $\db$} (2_2);
  \draw [->] (2_2) to node[pos=0.8, left] {\tiny $\db$} (3_2);
  % 0-th dbbar
  \draw [->] (0_0) to node[pos=0.8, left] {\tiny $\dbbar$} (0_1);
  \draw [->] (0_1) to node[pos=0.8, left] {\tiny $\dbbar$} (0_2);
  % 1-st dbbar
  \draw [->] (1_0) to node[pos=0.8, left] {\tiny $\dbbar$} (1_1);
  \draw [->] (1_1) to node[pos=0.3, left] {\tiny $\dbbar$} (1_2);
  % 2-nd dbbar
  \draw [->] (2_0) to node[pos=0.3, left] {\tiny $\dbbar$} (2_1);
  \draw [->] (2_1) to node[pos=0.8, left] {\tiny $\dbbar$} (2_2);
  % 3-rd dbbar
  \draw [->] (3_0) to node[pos=0.8, left] {\tiny $\dbbar$} (3_1);
  \draw [->] (3_1) to node[pos=0.8, left] {\tiny $\dbbar$} (3_2);
  % dhor
  \draw [->] (2_0) to node[pos=0.5, above] {\tiny $\dhor$} (3_0);
  \draw [->] (1_1) to node[pos=0.5, above] {\tiny $\dhor$} (2_1);
  \draw [->] (0_2) to node[pos=0.5, above] {\tiny $\dhor$} (1_2);
 \end{tikzpicture}
\end{equation}
The operator $L_{0,1}^{\dbbar} := i\db\dbbar\db \colon \sR^{0,1} \to \sR^{2,2}$ is the \defn{GJMS operator on $(0,1)$-forms};
it yields a detour complex connecting the Kohn--Rossi complex $\sR^{0,0} \to \sR^{0,1} \to \sR^{0,2}$ to the conjugate Kohn--Rossi complex $\sR^{1,2} \to \sR^{2,2} \to \sR^{3,2}$.
The (complexified) \defn{Rumin complex} is obtained by setting $\sR^k := \bigoplus_{p+q=k} \sR^{p,q}$ and recovering the exterior derivative as $d=\db+\dhor+\dbbar$.

We require two additional ingredients to define our $Q$-curvature operators.
The first is the construction of the bigraded Rumin complex in terms of true complex-valued differential forms~\cite{Case2021rumin}.
The benefit of this construction is that it leads to a balanced $A_\infty$-algebra structure on $\sR^{\bullet,\bullet} := \bigoplus_{p,q} \sR^{p,q}$.
In particular, there are products $\rwedge \colon \sR^k \times \sR^\ell \to \sR^{k+\ell}$ and $\krwedge \colon \sR^{p,q} \times \sR^{r,s} \to \sR^{p+r,q+s}$ which are associative up to homotopy.
The second is the \defn{Lee form}
\begin{equation*}
 \ell^\theta := -iE_{\alpha\bar\beta}\,\theta^\alpha \wedge \theta^{\bar\beta} - \nabla^{\bar\nu}E_{\alpha\bar\nu}\,\theta\wedge \theta^\alpha - \nabla^\mu E_{\mu\bar\beta}\,\theta \wedge \theta^{\bar\beta} 
\end{equation*}
associated to each contact form $\theta$, where $E_{\alpha\bar\beta}:=\frac{1}{4}\bigl( R_{\alpha\bar\beta} - \frac{R}{2}h_{\alpha\bar\beta}\bigr)$ is a multiple of the trace-free part of the pseudohermitian Ricci tensor.
Notably, $\ell^\theta \in \sS^1 \cap \ker d$, $\sS^1 := \Real \sR^{1,1}$.
Denote $\sS^0 := \Real\sR^{0,0}$, so that $i\db\dbbar$ maps $\sS^0$ into $\sS^1$.
The space $\sP$ of \defn{CR pluriharmonic functions} is~\cite{Lee1988} equal to $\ker i\db\dbbar \cap \sS^0$.
An important property of the Lee form is that the \defn{Lee class}
\begin{equation*}
 [\ell^\theta] \in H^1(M;\sP) := \frac{\ker d \cap \sS^1}{\im i\db\dbbar \cap \sS^1} ,
\end{equation*}
$\sS^0 := \Real\sR^{0,0}$ and $\sS^2 := \sR^3$, is independent of the choice of contact form~\cite{Case2021rumin} and vanishes if and only if $(M^5,T^{1,0})$ admits a pseudo-Einstein contact form~\cites{Case2021rumin,Lee1988}.

Our first result is the formula for and the transformation properties of the $Q$-curvature operator on $\dbbar$-closed $(0,1)$-forms.

\begin{theorem}
 \label{Q01-definition-and-properties}
 Let $(M^5,T^{1,0},\theta)$ be a pseudohermitian five-manifold.
 Define the operator $Q_{0,1}^{\dbbar} \colon \sR^{0,1} \to \sR^{2,2}$ by
 \begin{equation*}
  Q_{0,1}^{\dbbar}\omega := \hodge\db^\ast\left(\db\db^\ast + \dbbar\dbbar^\ast + \frac{R}{6}\krwedge\right)\db\omega + 2\ell^\theta \krwedge \db\omega - \db(\omega \krwedge \ell^\theta) \in \sR^{2,2} .
 \end{equation*}
 Then
 \begin{equation*}
  (Q_{0,1}^{\dbbar})^{e^\Upsilon\theta}\omega = (Q_{0,1}^{\dbbar})^\theta\omega + L_{0,1}^{\dbbar}(\Upsilon\omega)
 \end{equation*}
 for all $\omega \in \ker \dbbar \cap \sR^{0,1}$ and all $\Upsilon \in C^\infty(M)$.
\end{theorem}

Of relevance to the $Q$-curvature operator on $d$-closed $(1,1)$-forms is the CR pluriharmonic complex and associated detour complex
\begin{equation*}
 \begin{tikzpicture}[baseline=(current bounding box.center),yscale=0.8,xscale=1.0]
  \node (0) at (0,0)   {$\sS^0$};
  \node (1) at (2,0)   {$\sS^1$};
  \node (2) at (4,0)   {$\sS^2$};
  \node (3) at (6,0)   {$\sS^3$};
  \node (4) at (8,0)   {$\sS^4$,};
  % Complex
  \draw [->] (0)  to node[above] {\tiny $i\db\dbbar$} (1);
  \draw [->] (1)  to node[above] {\tiny $d$} (2);
  \draw [->] (2)  to node[above] {\tiny $d$} (3);
  \draw [->] (3)  to node[above] {\tiny $d$} (4);
  % Detour
  \draw [->] (k.south) -- (2,-1) to node[above] {\tiny $L_{1,1}^{\mathrm{BC}} := d\dhor$} (6,-1) -- (2n-k);
 \end{tikzpicture}
\end{equation*}
where $\sS^k := \Real \sR^{k+1}$ for $k\geq2$.
See \cref{sec:bc} for an explanation of the notation $L_{1,1}^{\mathrm{BC}}$.
The formula for and the transformation properties of the $Q$-curvature operator on $d$-closed $(1,1)$-forms are as follows:

\begin{theorem}
 \label{mQ01-definition-and-properties}
 Let $(M^5,T^{1,0},\theta)$ be a pseudohermitian five-manifold.
 Define the operator $Q_{1,1}^d \colon \sR^{1,1} \to \sR^{3,1} \oplus \sR^{2,2}$ by
 \begin{align*}
  Q_{1,1}^d\omega & := -i\omD\omega + i\mD\omega , \\
  \mD\omega & := \hodge \db^\ast \left( \db\db^\ast + \dbbar\dbbar^\ast + \frac{R}{6} \krwedge \right)\omega + 2\ell^\theta \krwedge \omega .
 \end{align*}
 Then
 \begin{equation*}
  (Q_{1,1}^d)^{e^\Upsilon\theta}\omega = (Q_{1,1}^d)^\theta\omega + L_{1,1}^{\mathrm{BC}}(\Upsilon\omega) .
 \end{equation*}
 for all $\omega \in \ker d \cap \sR^{1,1}$ and all $\Upsilon \in C^\infty(M)$.
\end{theorem}

Note that $Q_{0,1}^{\dbbar} \equiv \mD\db \mod \im \db$, indicating the close relationship between the $Q$-curvature operators on $\dbbar$-closed $(0,1)$-forms and on $d$-closed $(1,1)$-forms.

The $Q$-curvature operator on $d$-closed $(1,1)$-forms gives a simple formula for the scalar GJMS operator $L_{\mathrm{FH}}$ and the scalar $Q$-curvature $Q_{\mathrm{FH}}$ defined by Fefferman and Hirachi~\cite{FeffermanHirachi2003}.

\begin{theorem}
 \label{recover-scalar}
 Let $(M^5,T^{1,0},\theta)$ be a pseudohermitian five-manifold.
 Then
 \begin{align*}
  \frac{1}{8}\hodge L_{\mathrm{FH}} & = \Lscal := idQ_{1,1}^d\db\dbbar , \\
  \frac{1}{8}\hodge Q_{\mathrm{FH}} & = \Qscal := dQ_{1,1}^d\ell^\theta .
 \end{align*}
\end{theorem}

It follows readily from \cref{mQ01-definition-and-properties} that $\Lscal$ is CR invariant and
\begin{equation*}
 (\Qscal)^{e^\Upsilon\theta} = (\Qscal)^\theta + (\Lscal)^\theta \Upsilon
\end{equation*}
for all contact forms $\theta$ and $e^\Upsilon\theta$ on a pseudohermitian five-manifold.
\Cref{recover-scalar} also recovers two well-known facts about the scalar GJMS operator.
First~\cite{FeffermanHirachi2003}, it holds that $\sP \subset \ker \Lscal$.
Second~\cite{GoverGraham2005}, the identity $\Lscal = -2\db Q_{0,1}^{\dbbar}\dbbar$ implies that $\hodge \Lscal$ is formally self-adjoint.
In particular, if $(M^5,T^{1,0})$ is a closed CR five-manifold and $u\in\sP$, then $\int u \Qscal$ is independent of the choice of contact form.

Since the scalar $Q$-curvature is the constant term of the $P^\prime$-operator~\cites{CaseYang2012,Hirachi2013}, one should regard $\Lscal = L_{0,0}^{\dbbar}$ and $\Qscal = Q_{0,0}^{\dbbar}1$, though we do not try to define $Q_{0,0}^{\dbbar}$ in general in this article.

Following Branson and Gover~\cite{BransonGover2005}, we define the spaces of CR $\dbbar$-harmonic $(0,1)$-forms and CR $d$-harmonic $(1,1)$-forms by
\begin{align*}
 \HQ^{0,1} & := \left\{ \omega \in \sR^{0,1} \suchthatcolon \dbbar\omega = 0, \db Q_{0,1}^{\dbbar}\omega = 0 \right\} , \\
 \mHQ^{1,1} & := \left\{ \omega \in \sR^{1,1} \suchthatcolon d\omega = 0, dQ_{1,1}^d\omega = 0 \right\} ,
\end{align*}
respectively.
It is clear that the identity map induces CR invariant morphisms
\begin{equation*}
 \HQ^{0,1} \to H^{0,1}(M) := \frac{\ker \dbbar \cap \sR^{0,1}}{\im \dbbar \cap \sR^{0,1}}
\end{equation*}
and
\begin{equation*}
 \mHQ^{1,1} \to \HBC^{1,1}(M) := \frac{\ker d \cap \sR^{1,1}}{\im i\db\dbbar \cap \sR^{1,1}} .
\end{equation*}
The notation $\HBC^{1,1}(M)$ is inspired by the resemblance of this group with the corresponding Bott--Chern cohomology group in complex geometry~\cite{BottChern1965}.
Taking real parts also yields a morphism $\Real\mHQ^{1,1} \to H^1(M;\sP)$.
This latter morphism gives a cohomological characterization of CR five-manifolds which admit a \defn{$Q$-flat contact form};
i.e.\ a contact form $\theta$ for which $(\Qscal)^\theta = 0$.

\begin{theorem}
 \label{scalar-q-flat-characterization}
 An orientable CR five-manifold $(M^5,T^{1,0})$ admits a $Q$-flat contact form if and only if $[\ell^\theta] \in \im \bigl( \Real\mHQ^{1,1} \to H^1(M;\sP) \bigr)$.
\end{theorem}

We do not know if either $\HQ^{0,1} \to H^{0,1}(M)$ or $\mHQ^{1,1} \to \HBC^{1,1}(M)$ is surjective.
The surjectivity of the corresponding map~\cites{AubryGuillarmou2011,BransonGover2005} on conformal manifolds is also unknown in general.

There is another characterization~\cite{Takeuchi2020b} of CR manifolds which admit a $Q$-flat contact form:
A closed, strictly pseudoconvex manifold admits a $Q$-flat contact form if and only if $\Qscal$ annihilates $\ker \Lscal$.
It is presently known that $\Qscal$ annihilates constants~\cite{Marugame2017}.
We improve this latter result by proving that $\Qscal$ annihilates the space of CR pluriharmonic functions.
More precisely:

\begin{theorem}
 \label{q00-integral}
 Let $(M^5,T^{1,0},\theta)$ be a closed CR five-manifold.
 Then
 \begin{equation}
  \label{eqn:q00-integral}
  \int_M u \Qscal = \frac{\pi^2}{2}\lp [d_b^cu] \cup c_1^2(T^{1,0}), [M] \rp
 \end{equation}
 for all $u\in\sP$, where $d_b^cu := i\dbbar u - i\db u$ and $[M]$ is the fundamental class of $M$.
 In particular, if $(M^5,T^{1,0})$ is strictly pseudoconvex, then $\int u \Qscal = 0$ for all $u\in\sP$.
\end{theorem}

The final conclusion of \cref{q00-integral} follows from~\eqref{eqn:q00-integral} and the fact~\cite{Takeuchi2018} that $c_1^2(T^{1,0})=0$ in $H^4(M;\bR)$ on all closed, strictly pseudoconvex CR five-manifolds.
Since there are~\cite{Takeuchi2017} closed, strictly pseudoconvex CR manifolds with $\sP \subsetneq \ker \Lscal$, \cref{q00-integral} does not answer the question of whether all closed, strictly pseudoconvex CR manifolds admit $Q$-flat contact forms.

Yuya Takeuchi~\cite{Takeuchi2020c} has proven an analogue of~\eqref{eqn:q00-integral} on closed, strictly pseudoconvex $(2n+1)$-manifolds, with $c_1^2(T^{1,0})$ replaced by $c_1^n(T^{1,0})$.

One reason for the interest in $Q$-flat contact forms is that all pseudo-Einstein contact forms are $Q$-flat~\cite{FeffermanHirachi2003}.
Jack Lee conjectured~\cite{Lee1988} that every closed, strictly pseudoconvex CR manifold with trivial real first Chern class admits a pseudo-Einstein contact form.
Even the following weaker form of his conjecture is open:

\begin{conjecture}[Weak Lee Conjecture]
 \label{weak-lee-conjecture}
 Let $(M^5,T^{1,0})$ be a closed, strictly pseudoconvex five-manifold.
 If $c_1(T^{1,0})=0$ in $H^2(M;\bR)$, then $(M^5,T^{1,0})$ admits a $Q$-flat contact form.
\end{conjecture}

One can similarly formulate the Weak Lee Conjecture in general dimensions.
The Weak Lee Conjecture is only known in dimension three provided one makes the additional assumption of embeddability~\cite{Takeuchi2019}.

Of relevance to us is the fact~\cites{Case2021rumin,Garfield2001} that a CR manifold has trivial real first Chern class if and only if $[\ell^\theta] \in \im \bigl( H^{0,1}(M) \to H^1(M;\sP) \bigr)$.
\Cref{scalar-q-flat-characterization} then implies that the Weak Lee Conjecture holds on manifolds for which the morphism $\HQ^{0,1} \to H^{0,1}(M)$ is surjective.
Following Aubry and Guillarmou~\cite{AubryGuillarmou2011}, we prove surjectivity under the assumption that the $Q$-curvature operator on $\dbbar$-closed $(0,1)$-forms is nonnegative.

\begin{theorem}
 \label{q01-nonnegative}
 Let $(M^5,T^{1,0})$ be a closed, strictly pseudoconvex five-manifold such that
 \begin{equation}
  \label{eqn:q01-pairing}
  \int_M \omega \rwedge \overline{Q_{0,1}^{\dbbar}\tau} \geq 0
 \end{equation}
 for all $\omega, \tau \in \ker \dbbar \cap \sR^{0,1}$.
 Then $\HQ^{0,1} \to H^{0,1}(M)$ is surjective.
 In particular, if $c_1(T^{1,0})=0$ in $H^2(M;\bR)$, then $(M^5,T^{1,0})$ admits a $Q$-flat contact form.
\end{theorem}

The left-hand side of~\eqref{eqn:q01-pairing} defines a CR invariant Hermitian form on the space of $\dbbar$-closed $(0,1)$-forms.
\Cref{recover-scalar} implies that if~\eqref{eqn:q01-pairing} holds, then the scalar GJMS operator on functions is nonnegative.
In particular, there are examples~\cite{Takeuchi2017} of closed, strictly pseudoconvex five-manifolds which do not satisfy~\eqref{eqn:q01-pairing}.

The proof of \cref{q01-nonnegative} requires two main steps.

First, as a CR invariant analogue of the Kohn's Hodge isomorphism theorem~\cite{Kohn1965}, we identify a space $\HB^{0,1}$ with the property that $\dim \HQ^{0,1}/\HB^{0,1} = \dim H^{0,1}(M)$.
A key difficulty is that the space $\HQ^{0,1}$ of CR $\dbbar$-harmonic $(0,1)$-forms is infinite-dimensional.
In particular, $(\dbbar, \db Q_{0,1}^{\dbbar})$ is not graded injectively hypoelliptic;
i.e.\ $(\dbbar^\ast\dbbar)^3 + \dbbar\dbbar^\ast \hodge Q_{0,1}^{\dbbar}$ is not hypoelliptic (cf.\ \cite{BransonGover2005}).
We overcome this difficulty by using the fact that the $L^2$-realization of the scalar GJMS operator is self-adjoint and has closed range~\cite{Takeuchi2020b}.

Second, we show that if~\eqref{eqn:q01-pairing} holds, then $\HB^{0,1} = \ker \bigl( \HQ^{0,1} \to H^{0,1}(M) \bigr)$.
This is a straightforward adaptation of the argument~\cite{AubryGuillarmou2011} in the conformal setting.

This article is organized as follows:

In \cref{sec:bg} we recall relevant background material on CR and pseudohermitian manifolds.

In \cref{sec:invariance} we carry out most of the computations needed to verify the transformation laws of our $Q$-curvature operators.

In \cref{sec:bc} we prove \cref{mQ01-definition-and-properties} and study some properties of the morphism $\mHQ^{1,1} \to \HBC^{1,1}(M)$.

In \cref{sec:kr} we prove \cref{Q01-definition-and-properties} and study some properties of the morphism $\HQ^{0,1} \to H^{0,1}(M)$.

In \cref{sec:scalar} we relate the scalar GJMS operator and the scalar $Q$-curvature to our $Q$-curvature operators and prove \cref{recover-scalar,scalar-q-flat-characterization,q00-integral}.

In \cref{sec:hodge} we further develop the ``Hodge theory'' of $\HQ^{0,1}$.
In particular, we characterize closed, strictly pseudoconvex manifolds for which the morphism $\HQ^{0,1} \to H^{0,1}(M)$ is surjective.

In \cref{sec:strong-lee} we study the condition~\eqref{eqn:q01-pairing} and prove \cref{q01-nonnegative}.

\subsection*{Acknowledgements}
I thank Yuya Takeuchi for getting me interested in the problem of finding conditions which imply the existence of a $Q$-flat contact form;
and an anonymous referee for helpful comments which improved the readability of this article.
This work was partially supported by the Simons Foundation (Grant \#524601).

\section{Background}
\label{sec:bg}

In this section we collect some relevant background on CR and pseudohermitian manifolds, including a review of the recent construction~\cite{Case2021rumin} of the bigraded Rumin algebra.
Given the focus of this article, we only discuss this background in the context of five-dimensional manifolds.

\subsection{CR geometry}
\label{subsec:bg/ph}

A \defn{CR five-manifold} is a pair $(M^5,T^{1,0})$ of a connected, orientable, real five-manifold $M^5$ and a nondegenerate complex rank $2$ distribution $T^{1,0} \subset TM \otimes \bC$ such that  $T^{1,0} \cap \overline{T^{0,1}} = \{0\}$ and $[T^{1,0}, T^{1,0}] \subset T^{1,0}$.
Nondegeneracy means that if $\theta$ is any \defn{contact form} --- i.e.\ $\ker \theta = H := \Real (T^{1,0} \oplus \overline{T^{1,0}})$ --- then the \defn{Levi form} $(Z,W) \mapsto -i\,d\theta(Z,\oW)$ on $T^{1,0}$ is nondegenerate.

The assumptions of orientability and connectedness are made for technical convenience:
Our constructions are local and some cohomological statements do not require orientability.
We leave the details to the interested reader.

A \defn{pseudohermitian five-manifold} is a pair $(M^5,T^{1,0},\theta)$ of a CR five-manifold and a contact form $\theta$.
We say that $(M^5,T^{1,0},\theta)$ is \defn{strictly pseudoconvex} if its Levi form is positive definite.
Note that if $\theta$ and $\htheta$ are both strictly pseudoconvex contact forms on $(M^5,T^{1,0})$, then there is an $\Upsilon \in C^\infty(M)$ such that $\htheta = e^\Upsilon\theta$.

The \defn{Reeb vector field} $T$ of a pseudohermitian five-manifold $(M^5,T^{1,0},\theta)$ is the unique real vector field such that $\theta(T)=1$ and $d\theta(T,\cdot)=0$.
An \defn{admissible coframe} is a set $\{ \theta^\alpha \}_{\alpha=1}^2$ of local complex-valued differential one-forms which annihilate local sections of $\overline{T^{1,0}} \oplus \bC T$ and for which $\{ \theta^\alpha \rv_{T^{1,0}} \}$ forms a local coframe for $(T^{1,0})^\ast$.
It follows that the functions $\{ h_{\alpha\bar\beta} \}_{\alpha,\beta=1}^2$ defined by
\begin{equation*}
 d\theta = ih_{\alpha\bar\beta}\,\theta^\alpha \wedge \theta^{\bar\beta} ,
\end{equation*}
$\theta^{\bar\beta} := \overline{\theta^\beta}$, define an Hermitian matrix.
We use $h_{\alpha\bar\beta}$ to raise and lower indices in the usual way.
The \defn{connection one-forms} $\omega_\alpha{}^\gamma$ and \defn{pseudohermitian torsion} $A_{\alpha\gamma}$ are uniquely determined from an admissible coframe by the system
\begin{align*}
 d\theta^\alpha & = \theta^\mu \wedge \omega_\mu{}^\alpha + A^\alpha{}_{\bar\beta}\,\theta \wedge \theta^{\bar\beta} , \\
 dh_{\alpha\bar\beta} & = \omega_{\alpha\bar\beta} + \omega_{\bar\beta\alpha}, \\
 A_{\alpha\gamma} & = A_{\gamma\alpha} .
\end{align*}
The connection one-forms determine the \defn{Tanaka--Webster connection}~\cites{Tanaka1975,Webster1978} by
\begin{equation*}
 \nabla_X Z_\alpha := \omega_\alpha{}^\mu(X) Z_\mu
\end{equation*}
for all local sections $Z_\alpha$ of $T^{1,0}$ and $X$ of $TM\otimes\bC$.
The \defn{pseudohermitian curvature} $R_{\alpha\bar\beta\gamma\bar\sigma}$ is determined from the \defn{curvature forms} $\Omega_\alpha{}^\gamma := d\omega_\alpha{}^\gamma - \omega_\alpha{}^\mu \wedge \omega_\mu{}^\gamma$ by
\begin{equation*}
 \Omega_\alpha{}^\gamma \equiv R_\alpha{}^\gamma{}_{\rho\bar\sigma}\,\theta^\rho\wedge\theta^{\bar\sigma} \mod \theta, \theta^\rho\wedge\theta^\mu, \theta^{\bar\sigma} \wedge\theta^{\bar\nu}.
\end{equation*}
The \defn{pseudohermitian Ricci tensor} is $R_{\alpha\bar\beta} := R_{\alpha\bar\beta\mu}{}^\mu$ and the \defn{pseudohermitian scalar curvature} is $R := R_\mu{}^\mu$.

We need to know how the Tanaka--Webster connection, the pseudohermitian torsion, and the pseudohermitian curvature transform under change of contact form.
The latter is most clearly expressed in terms of the \defn{CR Schouten tensor}
\begin{equation*}
 P_{\alpha\bar\beta} := \frac{1}{4}\left( R_{\alpha\bar\beta} - \frac{R}{6}h_{\alpha\bar\beta} \right) .
\end{equation*}

\begin{lemma}[\citelist{ \cite{Lee1988}*{Lemma~2.4} \cite{GoverGraham2005}*{Equation~(2.7)} }]
 \label{transformation}
 Let $(M^5,T^{1,0},\theta)$ be a pseudohermitian five-manifold.
 If $\htheta = e^\Upsilon\theta$, then
 \begin{align*}
  \hP_{\alpha\bar\beta} & = P_{\alpha\bar\beta} - \frac{1}{2}\left(\Upsilon_{\alpha\bar\beta} + \Upsilon_{\bar\beta\alpha}\right) - \frac{1}{2}\Upsilon_\mu\Upsilon^\mu h_{\alpha\bar\beta} , \\
  \hA_{\alpha\gamma} & = A_{\alpha\gamma} + i\Upsilon_{\alpha\gamma} - i\Upsilon_\alpha\Upsilon_\gamma ,
 \end{align*}
 where $\Upsilon_\alpha:=\nabla_\alpha\Upsilon$ and $\Upsilon_{\alpha\bar\beta}:=\nabla_{\bar\beta}\nabla_\alpha\Upsilon$, and both the CR Schouten tensor $\hP_{\alpha\bar\beta}$ and the pseudohermitian torsion $\hA_{\alpha\gamma}$ of $\htheta$ are computed with respect to the admissible coframe $\{ \htheta^\alpha := \theta^\alpha + i\Upsilon^\alpha\theta \}$.
 Moreover, if $\omega_\alpha$ is a section of $(T^{1,0})^\ast$, then
 \begin{align*}
  \hnabla_\gamma \omega_\alpha & = \nabla_\gamma\omega_\alpha - \Upsilon_\gamma\omega_\alpha - \Upsilon_\alpha\omega_\gamma , \\
  \hnabla_{\bar\beta} \omega_\alpha & = \nabla_{\bar\beta}\omega_\alpha + \Upsilon^\mu\omega_\mu h_{\alpha\bar\beta} .
 \end{align*}
\end{lemma}

For notational convenience, we denote by
\begin{align*}
 E_{\alpha\bar\beta} & := P_{\alpha\bar\beta} - \frac{P}{2}h_{\alpha\bar\beta} , \\
 P & := P_\mu{}^\mu
\end{align*}
the trace-free part of and the trace of the CR Schouten tensor, respectively.

\subsection{The bigraded Rumin complex}
\label{subsec:bg/rumin}

We regard the \defn{bigraded Rumin complex}, depicted in~\eqref{eqn:bigraded-rumin-complex}, as a bigraded balanced $A_\infty$-algebra formed from subspaces of the space of complex-valued differential forms~\cite{Case2021rumin}.

We first define the spaces $\sR^{p,q}$ of \defn{$(p,q)$-forms} in the bigraded Rumin complex.
If $p+q\leq 2$, then
\begin{multline*}
 \sR^{p,q} := \biggl\{ \frac{1}{p!q!}\Bigl(\omega_{\Alpha\bar\Beta}\,\theta^\Alpha\wedge\theta^{\bar\Beta} - \frac{pi}{3-p-q}\nabla^{\mu}\omega_{\mu\Alpha^\prime\bar\Beta}\,\theta \wedge \theta^{\Alpha^\prime} \wedge \theta^{\bar\Beta} \\
  + \frac{(-1)^pqi}{3-p-q}\nabla^{\bar\nu}\omega_{\Alpha\bar\nu\bar\Beta^\prime}\,\theta \wedge \theta^{\Alpha} \wedge \theta^{\bar\Beta^\prime} \Bigr) \mathrel{}:\mathrel{} \omega_{\mu\Alpha^\prime}{}^\mu{}_{\bar\Beta^\prime} = 0, \lv A\rv = p, \rv B\rv = q \biggr\} ,
\end{multline*}
where $\Alpha = (\alpha_1,\dotsc,\alpha_p)$ and $\Beta = (\beta_1,\dotsc,\beta_q)$ are multi-indices of length $p$ and $q$, respectively, and we identify $\Alpha = (\alpha, \Alpha^\prime)$ and $\Beta = (\beta, \Beta^\prime)$ for multi-indices $\Alpha^\prime$ and $\Beta^\prime$ of length $p-1$ and $q-1$, respectively.
\Cref{transformation} implies that $\sR^{p,q}$, $p+q\leq 2$, is CR invariant.
Note that an element $\omega\in\sR^{p,q}$, $p+q\leq 2$, is determined by $\omega\rv_{H\otimes\bC}$.
For this reason, we abuse notation and identify $\omega \equiv \frac{1}{p!q!}\omega_{\Alpha\bar\Beta}\,\theta^{\Alpha} \wedge \theta^{\bar\Beta}$ in $\sR^{p,q}$, $p+q\leq 2$.
If $p+q\geq 3$, then
\begin{multline*}
 \sR^{p,q} := \biggl\{ \frac{1}{(p+q-3)!(3-p)!(2-q)!}\omega_{\Alpha\bar\Beta}\,\theta \wedge \theta^\Alpha \wedge \theta^{\bar\Beta} \wedge d\theta^{p+q-3} \\
  :\mathrel{} \omega_{\mu\Alpha^\prime}{}^\mu{}_{\bar\Beta^\prime}=0, \lv A\rv = 2-q, \lv B \rv = 3-p \biggr\} .
\end{multline*}
It is clear that $\sR^{p,q}$, $p+q\geq3$, is CR invariant.

We now define the operators $\db$, $\dhor$, and $\dbbar$ in the bigraded Rumin complex.
If $p+q\leq 1$, then
\begin{equation}
 \label{eqn:db-low}
 \begin{aligned}
  \db\omega & :\equiv \frac{1}{p!q!}\left( \nabla_\alpha\omega_{\Alpha\bar\Beta} - \frac{q}{3-p-q}h_{\alpha\bar\beta}\nabla^{\bar\nu}\omega_{\Alpha\bar\nu\bar\Beta^\prime} \right)\,\theta^{\alpha\Alpha}\wedge\theta^{\bar\Beta} && \text{in $\sR^{p+1,q}$} , \\
  \dbbar\omega & :\equiv \frac{(-1)^p}{p!q!} \left( \nabla_{\bar\beta}\omega_{\Alpha\bar\Beta} - \frac{p}{3-p-q}h_{\alpha\bar\beta}\nabla^\mu\omega_{\mu\Alpha^\prime\bar\Beta} \right) \, \theta^\Alpha \wedge \theta^{\bar\beta\bar\Beta} && \text{in $\sR^{p,q+1}$}
 \end{aligned}
\end{equation}
for all $\omega \equiv \frac{1}{p!q!}\omega_{\Alpha\bar\Beta}\,\theta^\Alpha \wedge \theta^{\bar\Beta}$ in $\sR^{p,q}$.
If $p+q=2$, then
\begin{align*}
 \db\omega & := \frac{(-1)^{p-1}i}{p!(q-1)!}\left(\nabla_\alpha\nabla^{\bar\nu}\omega_{\Alpha\bar\nu\bar\Beta^\prime} + iA_\alpha{}^{\bar\nu}\omega_{\Alpha\bar\nu\bar\Beta^\prime}\right)\,\theta \wedge \theta^{\alpha\Alpha} \wedge \theta^{\bar\Beta^\prime} , \\
 \dhor\omega & := \frac{1}{p!q!}\left( \nabla_0\omega_{\Alpha\bar\Beta} + pi\nabla_\alpha\nabla^\mu\omega_{\mu\Alpha^\prime\bar\Beta} - qi\nabla_{\bar\beta}\nabla^{\bar\nu}\omega_{\Alpha\bar\nu\bar\Beta^\prime} \right) \, \theta \wedge \theta^\Alpha \wedge \theta^{\Bar\Beta} , \\
 \dbbar\omega & := \frac{(-1)^{p-1}i}{(p-1)!q!} \left( \nabla_{\bar\beta}\nabla^\mu\omega_{\mu\Alpha^\prime\bar\Beta} - iA_{\bar\beta}{}^\mu\omega_{\mu\Alpha^\prime\bar\Beta} \right) \, \theta \wedge \theta^{\Alpha^\prime} \wedge \theta^{\bar\beta\bar\Beta}
\end{align*}
for all $\omega \equiv \frac{1}{p!q!}\omega_{\Alpha\bar\Beta}\,\theta^\Alpha\wedge\theta^{\bar\Beta}$ in $\sR^{p,q}$.
If $p+q\geq3$, then
\begin{equation}
 \label{eqn:db-high}
 \begin{split}
  \db\omega & := \frac{(-1)^qi}{(p+q-2)!(2-p)!(2-q)!}\nabla^{\bar\nu}\omega_{\Alpha\bar\nu\bar\Beta^\prime}\,\theta \wedge \theta^{\Alpha} \wedge \theta^{\bar\Beta^\prime} \wedge d\theta^{p+q-2} , \\
  \dbbar\omega & := \frac{-i}{(p+q-2)!(3-p)!(1-q)!}\nabla^\mu\omega_{\mu\Alpha^\prime\bar\Beta}\,\theta \wedge \theta^{\Alpha^\prime} \wedge \theta^{\bar\Beta} \wedge d\theta^{p+q-2}
 \end{split}
\end{equation}
for all $\omega = \frac{1}{(p+q-3)!(3-p)!(2-q)!}\omega_{\Alpha\bar\Beta}\,\theta \wedge \theta^\Alpha\wedge\theta^{\bar\Beta}\wedge d\theta^{p+q-3}$ in $\sR^{p,q}$.
The operators $\db$, $\dbbar$, and $\dhor$ are all CR invariant.
Moreover, if $k\not=2$, then the exterior derivative on $\sR^k := \bigoplus_{p+q=k}\sR^{p,q}$ is $d=\db+\dbbar$;
if $k=2$, then the exterior derivative on $\sR^k$ is $d=\db+\dhor+\dbbar$.
As noted by Rumin~\cite{Rumin1990}, the complexified de Rham cohomology groups are recovered from $d \colon \sR^\bullet \to \sR^\bullet$, $\sR^\bullet := \bigoplus_k\sR^k$, in the usual way:
\begin{equation*}
 H^k(M;\bC) := \frac{\ker d \cap \sR^k}{\im d \cap \sR^k} .
\end{equation*}
The de Rham cohomology groups are recovered by taking real parts.

That~\eqref{eqn:bigraded-rumin-complex} is a bigraded complex means that the sum of all compositions of two operators vanishes.
In particular, $\dbbar\circ\dbbar = 0$ and $\db \circ \db = 0$.
Thus the \defn{Kohn--Rossi cohomology groups}~\cite{KohnRossi1965}
\begin{equation*}
 H^{p,q}(M) := \frac{\ker \dbbar \cap \sR^{p,q}}{\im \dbbar \cap \sR^{p,q}}
\end{equation*}
and the \defn{conjugate Kohn--Rossi cohomology groups}
\begin{equation*}
 H_{\db}^{p,q}(M) := \frac{\ker \db \cap \sR^{p,q}}{\im \db \cap \sR^{p,q}}
\end{equation*}
are well-defined.
We also define
\begin{equation*}
 d_b^c :=
 \begin{cases}
  i\dbbar - i\db , & \text{on $\sR^{p,q}$, $p+q\not=2$} , \\
  -\dbbar + \dhor - \db , & \text{on $\sR^{p,q}$, $p+q=2$} .
 \end{cases}
\end{equation*}
Then $d_b^c=JdJ^{-1}$, where $J := i^{q-p}$ on $\sR^{p,q}$, $p+q \leq 2$, and $J:= i^{q-p+1}$ on $\sR^{p,q}$, $p+q \geq 3$.

We also require a few other cohomology groups.
To that end, set $\sS^0:=\sR^0$ and $\sS^1:=\Real\sR^{1,1}$.
Since~\eqref{eqn:bigraded-rumin-complex} is a bigraded complex, the operator $i\db\dbbar$ maps $\sS^0$ to $\sS^1$ and satisfies $d \circ i\db\dbbar = 0$.
Therefore
\begin{equation*}
 H^1(M;\sP) := \frac{\ker d \cap \sS^{1}}{\im i\db\dbbar \cap \sS^{1}}
\end{equation*}
is well-defined.
Note that $\ker i\db\dbbar \cap \sS^0$ is the space $\sP$ of \defn{CR pluriharmonic functions}~\cite{Lee1988} and that if $(M^5,T^{1,0})$ is strictly pseudoconvex and locally embeddable, then $H^1(M;\sP)$ coincides with the corresponding cohomology group with coefficients in the sheaf of CR pluriharmonic functions~\cite{Case2021rumin}*{Theorem~9.3}.

Of particular note is the \defn{Lee form}
\begin{equation}
 \label{eqn:lee-form}
 \ell^\theta :\equiv -iE_{\alpha\bar\beta}\,\theta^\alpha\wedge\theta^{\bar\beta} \quad \text{in $\sS^1$} .
\end{equation}
Note that $\ell^\theta$ depends on the choice of contact form.
It holds~\cites{Lee1988,Case2021rumin} that the Lee form is $d$-closed and that the corresponding cohomology class $[\ell^\theta] \in H^2(M;\bR)$ is CR invariant;
indeed~\cite{Case2021rumin}*{Lemma~19.7},
\begin{equation}
 \label{eqn:lee-class-in-H2}
 [\ell^\theta] = -\frac{\pi}{2} c_1(T^{1,0}) .
\end{equation}
In fact, the cohomology class $[\ell^\theta] \in H^1(M;\sP)$ is CR invariant;
this is a consequence of the following lemma.

\begin{lemma}[\cite{Case2021rumin}*{Lemma~19.3}]
 \label{lee-form-transformation}
 Let $(M^5,T^{1,0})$ be a CR five-manifold.
 Then
 \begin{equation*}
  \ell^{e^\Upsilon\theta} = \ell^\theta + i\db\dbbar\Upsilon
 \end{equation*}
 for all contact forms $\theta$ and $e^\Upsilon\theta$ on $(M^5,T^{1,0})$.
\end{lemma}

We say that $(M^5,T^{1,0},\theta)$ is \defn{pseudo-Einstein} if $\ell^\theta=0$.
One readily checks that a CR five-manifold $(M^5,T^{1,0})$ admits a pseudo-Einstein contact form if and only if $[\ell^\theta]=0$ in $H^1(M;\sP)$~\citelist{ \cite{Lee1988}*{Proposition~5.2} \cite{Case2021rumin}*{Lemma~19.6} }.

The complexifications of $H^1(M;\sP)$ and $H^0(M;\sP):=\sP$ are
\begin{align*}
 \HBC^{1,1}(M) & := \frac{\ker d \cap \sR^{1,1}}{\im \db\dbbar \cap \sR^{1,1}} , \\
 \HA^{0,0}(M) & := \ker \db\dbbar \cap \sR^{0,0} ,
\end{align*}
respectively.
This notation anticipates the relevance of the $Q$-curvature operator on $d$-closed forms to the CR analogues of the Bott--Chern~\cite{BottChern1965} and Aeppli~\cite{Aeppli1965} cohomology groups, respectively;
see \cref{sec:bc} for further details.

There is a canonical CR invariant projection $\pi$ from the space $\COmega^kM$ of complex-valued differential $k$-forms to $\sR^k$.
Let $\omega \in \COmega^kM$ and let $\theta$ be a contact form.
The fact that $d\theta\rv_H$ is a symplectic form on $H$ implies that~\cite{Case2021rumin}*{Theorem~3.7}
\begin{enumerate}
 \item if $k \leq 2$, then there is a unique $\theta \wedge \xi \in \theta \wedge \COmega^{k-2}M$ such that
 \begin{equation*}
  \theta \wedge \omega \wedge d\theta^{3-k} = \theta \wedge \xi \wedge d\theta^{4-k} ;
 \end{equation*}
 \item if $k \geq 3$, then there is a unique $\theta \wedge \xi \in \theta \wedge \COmega^{4-k}M$ such that
 \begin{equation*}
  \theta \wedge \omega = \theta \wedge \xi \wedge d\theta^{k-2} .
 \end{equation*}
\end{enumerate}
Moreover, the map $\Gamma \colon \COmega^kM \to \theta \wedge \COmega^{k-2}M$,
\begin{equation*}
 \Gamma\omega :=
 \begin{cases}
  \theta \wedge \xi , & \text{if $k \leq 2$}, \\
  \theta \wedge \xi \wedge d\theta^{k-3} , & \text{if $k \geq 3$} ,
 \end{cases}
\end{equation*}
is CR invariant.
The map
\begin{equation*}
 \pi\omega := \omega - d\Gamma\omega - \Gamma d\omega
\end{equation*}
is the aforementioned CR invariant projection $\pi \colon \COmega^kM \to \sR^k$.
This projection yields products which give $\sR^\bullet$ and $\sR^{\bullet,\bullet} := \bigoplus_{p,q} \sR^{p,q}$ the structure of balanced $A_\infty$-algebras.
The relevant properties of these structures are as follows:

Define $\rwedge \colon \sR^\bullet \times \sR^\bullet \to \sR^\bullet$ by
\begin{equation*}
 \omega \rwedge \tau := \pi\left( \omega \wedge \tau \right) .
\end{equation*}
Let $\omega \in \sR^{p,q}$ and $\tau \in \sR^{r,s}$.
If $p+q+r+s\leq 2$ or $\max\{p+q,r+s\}\geq3$, then $\omega \rwedge \tau \in \sR^{p+r,q+s}$;
otherwise $\omega \rwedge \tau \in \sR^{p+r+1,q+s-1} \oplus \sR^{p+r,q+s}$~\cite{Case2021rumin}*{Lemma~8.12}.
For example, it is straightforward to compute that
\begin{multline}
 \label{eqn:rwedge-11-11}
 \omega \rwedge \tau = \left(\frac{1}{2}\nabla_\alpha(\omega_{\mu\bar\nu}\tau^{\bar\nu\mu}) + \omega_\alpha{}^\mu\nabla^{\bar\nu}\tau_{\mu\bar\nu} + \tau_\alpha{}^\mu\nabla^{\bar\nu}\omega_{\mu\bar\nu} \right) \, \theta \wedge \theta^\alpha \wedge d\theta \\
  + \left(\frac{1}{2}\nabla_ {\bar\beta}(\omega_{\mu\bar\nu}\tau^{\bar\nu\mu}) + \omega^{\bar\nu}{}_{\bar\beta}\nabla^{\mu}\tau_{\mu\bar\nu} + \tau ^{\bar\nu}{}_{\bar\beta}\nabla^{\mu}\omega_{\mu\bar\nu} \right) \, \theta \wedge \theta^{\bar\beta} \wedge d\theta
\end{multline}
for $\omega \equiv \omega_{\alpha\bar\beta}\,\theta^{\alpha}\wedge\theta^{\bar\beta}$ and $\tau \equiv \tau_{\alpha\bar\beta}\,\theta^{\alpha} \wedge \theta^{\bar\beta}$ in $\sR^{1,1}$;
and that
\begin{multline}
 \label{eqn:rwedge-11-01}
 \omega \rwedge \rho =  i \left( 2\rho_{\bar\beta}\nabla^\mu\omega_{\mu\bar\sigma} + \omega^{\bar\nu}{}_{\bar\sigma}\nabla_{\bar\beta}\rho_{\bar\nu} \right) \, \theta \wedge \theta^{\bar\beta} \wedge \theta^{\bar\sigma} \\
  + i\left( \frac{1}{2}\omega_{\alpha\bar\beta}\nabla^{\bar\nu}\rho_{\bar\nu} + \nabla_\alpha(\omega^{\bar\nu}{}_{\bar\beta}\rho_{\bar\nu}) - \rho_{\bar\beta}\nabla^{\bar\nu}\omega_{\alpha\bar\nu} - \frac{1}{2}h_{\alpha\bar\beta}\omega^{\bar\nu\mu}\nabla_\mu\rho_{\bar\nu} \right)\,\theta \wedge \theta^\alpha \wedge \theta^{\bar\beta}
\end{multline}
for all $\omega \equiv \omega_{\alpha\bar\beta}\,\theta^{\alpha} \wedge \theta^{\bar\beta}$ in $\sR^{1,1}$ and all $\rho \equiv \rho_{\bar\beta}\,\theta^{\bar\beta}$ in $\sR^{0,1}$.

We define $\krwedge \colon \sR^{p,q} \times \sR^{r,s} \to \sR^{p+r,q+s}$ by
\begin{equation*}
 \omega \krwedge \tau := \pi^{p+r,q+s}(\omega \rwedge \tau) ,
\end{equation*}
where $\pi^{p,q} \colon \sR^{\bullet,\bullet} \to \sR^{p,q}$ is the canonical projection.
The key properties of these products are summarized in the following lemma:

\begin{lemma}[\citelist{ \cite{Case2021rumin}*{Theorems~8.8 and~8.16} }]
 \label{leibnitz}
 Let $(M^5,T^{1,0})$ be a CR five-manifold.
 If $\omega\in\sR^{p,q}$ and $\tau\in\sR^{r,s}$, then
 \begin{align*}
  d(\omega \rwedge \tau) & = d\omega \rwedge \tau + (-1)^{p+q}\omega \rwedge d\tau , \\
  \dbbar (\omega \krwedge \tau) & = \dbbar\omega \krwedge \tau + (-1)^{p+q}\omega \krwedge \dbbar\tau .
 \end{align*}
 Moreover, the cup products induced by $\rwedge$ and $\krwedge$ on $H^\bullet(M;\bC)$ and $H^{\bullet,\bullet}(M)$, respectively, are associative.
\end{lemma}

We conclude this section with a review of the Hodge decomposition theorem for $\sR^{p,q}$ and some applications to the Kohn--Rossi cohomology groups.

Let $(M^5,T^{1,0},\theta)$ be a pseudohermitian five-manifold.
The Levi form induces an Hermitian inner product on $\sR^{p,q}$, $p+q\leq 2$, by
\begin{equation*}
 \lp \omega, \tau \rp := \frac{1}{p!q!}\omega_{\Alpha\bar\Beta}\otau^{\bar\Beta\Alpha} ,
\end{equation*}
where $\omega \equiv \frac{1}{p!q!}\omega_{\Alpha\bar\Beta}\,\theta^\Alpha\wedge\theta^{\bar\Beta}$ and $\tau \equiv \frac{1}{p!q!}\tau_{\Alpha\bar\Beta}\,\theta^\Alpha \wedge \theta^{\bar\Beta}$ in $\sR^{p,q}$, and $\otau_{\Beta\bar\Alpha} := \overline{\tau_{\Alpha\bar\Beta}}$.
We denote $\lv\omega\rv^2 := \lp\omega,\omega\rp$.
The \defn{$L^2$-inner product} is
\begin{equation*}
 \llp \omega, \tau \rrp := \frac{1}{2}\int_M \lp \omega, \tau \rp \, \theta \wedge d\theta^2 . 
\end{equation*}
The Hermitian inner product and the $L^2$-inner product on $\sR^{p,q}$, $p+q\geq3$, are defined similarly.
The Hodge star operator is defined by
\begin{equation*}
 \omega \rwedge \hodge\otau = \frac{1}{2}\lp \omega, \tau\rp\,\theta \wedge d\theta^2
\end{equation*}
and satisfies $\hodge^2=1$.
If $p+q\leq 2$, then~\cite{Case2021rumin}*{Lemma~10.5}
\begin{equation}
 \label{eqn:hodge-formula}
 \hodge \omega = \frac{(-1)^{\frac{(p+q)(p+q+1)}{2}}}{(2-p-q)!}\theta \wedge J\omega \wedge d\theta^{2-p-q} \in \sR^{3-q,2-p}
\end{equation}
for all $\omega\in\sR^{p,q}$.
We define
\begin{align*}
 \db^\ast & := (-1)^{p+q}\hodge \dbbar \hodge , \\
 \dbbar^\ast & := (-1)^{p+q} \hodge \db \hodge
\end{align*}
on $\sR^{p,q}$.
If $\omega \equiv \frac{1}{p!q!}\omega_{\Alpha\bar\Beta}\,\theta^\Alpha\wedge\theta^{\bar\Beta}$ in $\sR^{p,q}$, $p+q\leq 2$, then~\cite{Case2021rumin}*{Lemma~10.11}
\begin{equation}
 \label{eqn:dbbar-adjoint-formula}
 \begin{aligned}
  \db^\ast\omega & \equiv -\frac{1}{(p-1)!q!}\nabla^\mu\omega_{\mu\Alpha^\prime\bar\Beta}\,\theta^{\Alpha^\prime} \wedge \theta^{\bar\Beta} && \text{in $\sR^{p-1,q}$} , \\
  \dbbar^\ast\omega & \equiv -\frac{(-1)^p}{p!(q-1)!}\nabla^{\bar\nu}\omega_{\Alpha\bar\nu\bar\Beta^\prime}\,\theta^{\Alpha}\wedge\theta^{\bar\Beta^\prime} && \text{in $\sR^{p,q-1}$} .
 \end{aligned}
\end{equation}

The \defn{Kohn Laplacian} $\Box_b \colon \sR^{p,q} \to \sR^{p,q}$ is
\begin{align*}
 \Box_b & := \frac{2-p-q}{3-p-q}\dbbar\dbbar^\ast + \dbbar^\ast\dbbar , && \text{if $p+q\leq 1$} , \\
 \Box_b & := \dbbar^\ast\dbbar + \dbbar(\Box_b + \oBox_b)\dbbar^\ast , && \text{if $p+q = 2$} , 
\end{align*}
and satisfies $\Box_b\hodge\oomega = \hodge \overline{\Box_b\omega}$.
Note that $\Box_b$ is formally self-adjoint with respect to the $L^2$-inner product and that $\ker\Box_b = \ker\dbbar \cap \ker \dbbar^\ast$.
On closed, strictly pseudoconvex, pseudohermitian five-manifolds, $\Box_b$ is hypoelliptic on $\sR^{p,q}$ if and only if $q=1$~\cite{Case2021rumin}*{Propositions~13.1 and 13.4}.
Nevertheless, we have the following Hodge decomposition theorem valid in all bidegrees.

\begin{theorem}[\cite{Case2021rumin}*{Theorem~14.3 and Corollary~14.6}]
 \label{hodge}
 Let $(M^5,T^{1,0},\theta)$ be a closed, strictly pseudoconvex, pseudohermitian five-manifold.
 Then
 \begin{equation*}
  \sR^{p,q} = \ker\Box_b \oplus \im\dbbar \oplus \im\dbbar^\ast ,
 \end{equation*}
 where each summand on the right-hand side is understood as a subspace of $\sR^{p,q}$.
 In particular, the identity map induces an isomorphism $\ker\Box_b \cap \sR^{p,q} \cong H^{p,q}(M)$.
 Moreover, if $q=1$, then $\ker\Box_b$ is finite-dimensional.
\end{theorem}

We recover Serre duality from \cref{hodge} in the usual way.
It is convenient to record this in an unorthodox manner.

\begin{corollary}[Serre Duality]
 Let $(M^5,T^{1,0})$ be a closed, strictly pseudoconvex five-manifold.
 The map $(\omega, \tau) \mapsto \int \omega \rwedge \otau$ induces a perfect pairing
 \begin{equation*}
  \kJ \colon H^{0,1}(M) \times H_{\db}^{2,2}(M) \to \bC
 \end{equation*}
\end{corollary}

\begin{proof}
 Stokes' Theorem implies that if $f\in\sR^{0,0}$ and $\tau \in \ker\db \cap \sR^{2,2}$, then
 \begin{equation*}
  \int_M \dbbar f \rwedge \otau = \int_M df \rwedge \otau = -\int_M f \rwedge d\otau = -\int_M f \rwedge \overline{\db\tau} = 0 .
 \end{equation*}
 Similarly, if $\omega \in \ker\dbbar \cap \sR^{0,1}$ and $\rho \in \sR^{1,2}$, then $\int \omega \rwedge \overline{\db\rho} = 0$.
 Therefore $\kJ$ is well-defined.
 
 \Cref{hodge} implies that $H^{0,1}(M)$ and $H_{\db}^{2,2}(M)$ are finite-dimensional.
 Thus it suffices to show that $\kJ$ is nondegenerate.
 Suppose that $[\omega]\in H^{0,1}(M)$ is such that $\kJ([\omega],[\tau])=0$ for all $[\tau] \in H_{\db}^{2,2}(M)$.
 By \cref{hodge}, we may assume that $\omega \in \ker\Box_b$.
 Then $\hodge\omega \in \ker \db \cap \sR^{2,2}$.
 Since
 \begin{equation*}
  0 = \kJ( [\omega], [\hodge\omega] ) = \int_M \omega \rwedge \hodge\oomega = \frac{1}{2}\int_M \lv\omega\rv^2\,\theta \wedge d\theta^2 ,
 \end{equation*}
 we conclude that $[\omega]=0$.
 A similar argument shows that if $[\tau] \in H_{\db}^{2,2}(M)$ is such that $\kJ([\omega],[\tau])=0$ for all $[\omega] \in H^{0,1}(M)$, then $[\tau]=0$.
\end{proof}

\section{The operator $\mD$}
\label{sec:invariance}

Our $Q$-curvature operators are both built from the following operator.

\begin{definition}
 Let $(M^5,T^{1,0},\theta)$ be a pseudohermitian five-manifold.
 We define $\mD \colon \sR^{1,1} \to \sR^{2,2}$ by
 \begin{equation*}
  \mD\omega := \hodge \db^\ast \left( \db\db^\ast + \dbbar\dbbar^\ast + P \rwedge \right)\omega + 2\ell^\theta \krwedge \omega .
 \end{equation*}
\end{definition}

Equations~\eqref{eqn:db-low}, \eqref{eqn:rwedge-11-11}, \eqref{eqn:hodge-formula}, and~\eqref{eqn:dbbar-adjoint-formula} imply that if $\omega\equiv\omega_{\alpha\bar\beta}\,\theta^\alpha\wedge\theta^{\bar\beta}$ in $\sR^{1,1}$, then
\begin{multline}
 \label{eqn:mD-frame}
 \mD\omega = -i\Bigl[ \nabla^\gamma\left( \nabla_{\bar\beta}\nabla^{\bar\nu}\omega_{\gamma\bar\nu} + \nabla_\gamma\nabla^\mu\omega_{\mu\bar\beta} - h_{\gamma\bar\beta}\nabla^\mu\nabla^{\bar\nu}\omega_{\mu\bar\nu} - P\omega_{\gamma\bar\beta} \right) \\
  + \nabla_{\bar\beta}\left( E^{\bar\nu\mu}\omega_{\mu\bar\nu}\right) + 2E^{\bar\nu}{}_{\bar\beta}\nabla^\mu\omega_{\mu\bar\nu} + 2\omega^{\bar\nu}{}_{\bar\beta}\nabla^\mu E_{\mu\bar\nu} \Bigr] \, \theta \wedge \theta^{\bar\beta} \wedge d\theta .
\end{multline}

In this section we compute the transformation law for $\mD$ under change of contact form.
We begin by computing the transformation law for the second-order factor in the first summand of $\mD$.

\begin{lemma}
 \label{second-order-transformation}
 Let $(M^5,T^{1,0},\theta)$ be a pseudohermitian five-manifold.
 Consider the operator $D\colon\sR^{1,1} \to \sR^{2,1}$,
 \begin{equation*}
  D := \hodge \left( \db\db^\ast + \dbbar\dbbar^\ast + P \rwedge \right) .
 \end{equation*}
 Then
 \begin{equation*}
  D^{e^\Upsilon\theta}(\omega) = D^\theta(\omega) - \pi^{2,1}\left( d_b^c\Upsilon \rwedge \omega \right) 
 \end{equation*}
 for all $\omega\in\sR^{1,1}$ and all $\Upsilon \in C^\infty(M)$.
\end{lemma}

\begin{proof}
 Equations~\eqref{eqn:db-low}, \eqref{eqn:hodge-formula}, and~\eqref{eqn:dbbar-adjoint-formula} imply that
 \begin{equation*}
  D\omega = \left( \nabla_\alpha\nabla^\mu\omega_{\mu\bar\beta} + \nabla_{\bar\beta}\nabla^{\bar\nu}\omega_{\alpha\bar\nu} - h_{\alpha\bar\beta}\nabla^\mu\nabla^{\bar\nu}\omega_{\mu\bar\nu} - P\omega_{\alpha\bar\beta} \right) \, \theta \wedge \theta^\alpha \wedge \theta^{\bar\beta} ,
 \end{equation*}
 for all $\omega \equiv \omega_{\alpha\bar\beta}\,\theta^\alpha \wedge \theta^{\bar\beta}$ in $\sR^{1,1}$.
 On the one hand, a direct computation using \cref{transformation}, commutator identities~\cite{Lee1988}*{Lemma~2.3}, and the fact $\omega_\mu{}^\mu=0$ yields
 \begin{align*}
  \left.\frac{\partial}{\partial t}\right|_{t=0} D^{e^{t\Upsilon}\theta}\omega & = \Bigl( \nabla_\alpha(\Upsilon^\mu\omega_{\mu\bar\beta}) - \Upsilon_\alpha\nabla^\mu\omega_{\mu\bar\beta}  + \nabla_{\bar\beta}(\Upsilon^{\bar\nu}\omega_{\alpha\bar\nu}) - \Upsilon_{\bar\beta}\nabla^{\bar\nu}\omega_{\alpha\bar\nu} \\
   & \quad - \Upsilon^{\bar\nu\mu}\omega_{\mu\bar\nu}h_{\alpha\bar\beta} + \frac{1}{2}(\Upsilon_\mu{}^\mu + \Upsilon^\mu{}_\mu)\omega_{\alpha\bar\beta} \Bigr)\,\theta \wedge \theta^\alpha \wedge \theta^{\bar\beta} .
 \end{align*}
 On the other hand, \eqref{eqn:rwedge-11-01} yields
 \begin{multline*}
  \pi^{2,1}(i\db\Upsilon \rwedge \omega) \\
   = \left( \frac{1}{2} \Upsilon_\mu{}^\mu \omega_{\alpha\bar\beta} + \nabla_{\bar\beta}(\Upsilon^{\bar\nu}\omega_{\alpha\bar\nu}) - \Upsilon_\alpha\nabla^\mu\omega_{\mu\bar\beta} - \frac{1}{2} \Upsilon^{\bar\nu\mu}\omega_{\mu\bar\nu}h_{\alpha\bar\beta} \right)\,\theta \wedge \theta^\alpha \wedge \theta^{\bar\beta} .
 \end{multline*}
 The conclusion follows from the CR invariance of $d_b^c$ and the fact~\cite{Case2021rumin}*{Lemma~5.2} that $\overline{\pi^{2,1}\tau}=\pi^{2,1}\otau$.
\end{proof}

We now compute the transformation law for $\mD$.

\begin{proposition}
 \label{fourth-order-transformation}
 Let $(M^5,T^{1,0},\theta)$ be a pseudohermitian five-manifold.
 Then
 \begin{equation*}
  \mD^{e^\Upsilon\theta}\omega = \mD^\theta\omega + i\db\dbbar\left( \Upsilon \rwedge \omega \right) - i\db\left( \Upsilon \rwedge \dbbar\omega \right) + i\dbbar\Upsilon \rwedge \dhor\omega - i\db\Upsilon \rwedge \dbbar\omega 
 \end{equation*}
 for all $\omega\in\sR^{1,1}$ and all $\Upsilon \in C^\infty(M)$.
\end{proposition}

\begin{proof}
 Note that
 \begin{equation*}
  \mD\omega = \dbbar D\omega + 2\ell^\theta \krwedge \omega .
 \end{equation*}
 It follows from \cref{lee-form-transformation,second-order-transformation} that
 \begin{equation*}
  \mD^{e^\Upsilon\theta}\omega = \mD^\theta\omega - i\dbbar \pi^{2,1} \left( \dbbar\Upsilon \rwedge \omega - \db\Upsilon \rwedge \omega \right) + 2i\db\dbbar\Upsilon \krwedge \omega .
 \end{equation*}
 On the one hand, since $\pi^{2,1}(\db\Upsilon \rwedge \omega)=\db\Upsilon \krwedge \omega$, we deduce from \cref{leibnitz} that
 \begin{equation*}
  \dbbar \pi^{2,1} \left( \db\Upsilon \rwedge \omega \right) = \dbbar\db\Upsilon \krwedge \omega - \db\Upsilon \krwedge \dbbar\omega .
 \end{equation*}
 On the other hand, since $\pi^{2,1}(\dbbar\Upsilon \rwedge \omega) = \dbbar\Upsilon \rwedge \omega - \dbbar\Upsilon \krwedge\omega$, we compute from \cref{leibnitz} that
 \begin{align*}
  \dbbar \pi^{2,1} \left( \dbbar\Upsilon \rwedge \omega \right) & = \pi^{2,2}d\left( \dbbar\Upsilon \rwedge \omega \right) - \db\left( \dbbar\Upsilon \krwedge \omega \right) \\
   & = \pi^{2,2}\left( \db\dbbar\Upsilon \rwedge \omega - \dbbar\Upsilon \rwedge d\omega \right) - \db\dbbar \left( \Upsilon \krwedge \omega \right) + \db \left( \Upsilon \krwedge \dbbar\omega \right) \\
   & = \db\dbbar\Upsilon \krwedge \omega - \dbbar\Upsilon \krwedge \dhor\omega - \db\dbbar \left( \Upsilon \krwedge \omega \right) + \db \left( \Upsilon \krwedge \dbbar\omega \right) .
 \end{align*}
 It readily follows that
 \begin{equation}
  \label{eqn:mD-with-krwedge}
  \mD^{e^\Upsilon\theta}\omega = \mD^\theta\omega + i\db\dbbar\left( \Upsilon \krwedge \omega \right) - i\db\left( \Upsilon \krwedge \dbbar\omega \right) + i\dbbar\Upsilon \krwedge \dhor\omega - i\db\Upsilon \krwedge \dbbar\omega .
 \end{equation}
 The final conclusion follows from the fact~\cite{Case2021rumin}*{Lemma~8.12} that $\omega \rwedge \tau = \omega \krwedge \tau$ for all $\omega\in\sR^{p,q}$ and $\tau\in\sR^{r,s}$ such that $p+q+r+s\leq2$ or $\max\{p+q,r+s\}\geq3$.
\end{proof}

\section{The $Q$-curvature operator on $d$-closed $(1,1)$-forms}
\label{sec:bc}

In this section we introduce the $Q$-curvature operator on $d$-closed $(1,1)$-forms and establish some of its basic properties.
This operator is obtained by restriction from a multiple of the imaginary part of $\mD$.

\begin{definition}
 \label{defn:mQ}
 Let $(M^5,T^{1,0},\theta)$ be a pseudohermitian five-manifold.
 We define $Q_{1,1}^d \colon \sR^{1,1} \to \sR^{3,1} \oplus \sR^{2,2}$ by
 \begin{equation*}
  Q_{1,1}^d := -i\omD + i\mD ,
 \end{equation*}
 where $\omD\omega := \overline{\mD\oomega}$.
\end{definition}

The transformation law for $Q_{1,1}^d$ is linear in the conformal factor.

\begin{lemma}
 \label{mQ-transformation}
 Let $(M^5,T^{1,0},\theta)$ be a pseudohermitian five-manifold.
 Then
 \begin{equation*}
  (Q_{1,1}^d)^{e^\Upsilon\theta}\omega = (Q_{1,1}^d)^\theta\omega + d\dhor ( \Upsilon \rwedge \omega ) + d \left( \Upsilon \rwedge (\db\omega + \dbbar\omega) \right) - d\Upsilon \rwedge d_b^c\omega
 \end{equation*}
 for all $\omega\in\sR^{1,1}$ and all $\Upsilon \in C^\infty(M)$.
\end{lemma}

\begin{proof}
 This follows immediately from \cref{fourth-order-transformation}.
\end{proof}

Restricting \cref{mQ-transformation} to $d$-closed $(1,1)$-forms yields \cref{mQ01-definition-and-properties}.

\begin{proof}[Proof of \cref{mQ01-definition-and-properties}]
 This follows immediately from \cref{mQ-transformation}.
\end{proof}

\Cref{mQ01-definition-and-properties} motivates the following definition:

\begin{definition}
 Let $(M^5,T^{1,0},\theta)$ be a pseudohermitian five-manifold.
 The \defn{$Q$-curvature operator on $d$-closed $(1,1)$-forms} is the restriction
 \begin{equation*}
  Q_{1,1}^d \colon \ker d \cap \sR^{1,1} \to \sR^{3,1} \oplus \sR^{2,2} .
 \end{equation*}
 The \defn{Bott--Chern GJMS operator on $(1,1)$-forms} is
 \begin{equation*}
  L_{1,1}^{\mathrm{BC}} := d\dhor \colon \sR^{1,1} \to \sR^{3,1} \oplus \sR^{2,2} .
 \end{equation*}
\end{definition}

We use the adjective Bott--Chern to emphasize the context --- $Q$-curvature operators on representatives of elements of $\HBC^{1,1}(M)$ --- in which $L_{1,1}^{\mathrm{BC}}$ arises.

\Cref{mQ01-definition-and-properties} implies that if $\omega\in\sR^{1,1}$ is $d$-closed, then the condition $dQ_{1,1}^d\omega=0$ is CR invariant.
Following Branson and Gover~\cite{BransonGover2005}, we introduce the CR invariant vector space of CR $d$-harmonic $(1,1)$-forms.

\begin{definition}
 Let $(M^5,T^{1,0})$ be a CR five-manifold.
 The space of \defn{CR $d$-harmonic $(1,1)$-forms} is
 \begin{equation*}
  \mHQ^{1,1} := \left\{ \omega \in \sR^{1,1} \suchthatcolon d\omega = 0, dQ_{1,1}^d\omega = 0 \right\} .
 \end{equation*}
\end{definition}

There is a long exact sequence which partially describes the canonical morphism $\mHQ^{1,1} \to \HBC^{1,1}(M)$, $\omega \mapsto [\omega]$ (cf.\ \cite{BransonGover2005}*{Proposition~2.5}).

\begin{proposition}
 \label{sP-long-exact-sequence}
 Let $(M^5,T^{1,0})$ be a CR five-manifold.
 Then the sequence
 \begin{equation}
  \label{eqn:sP-long-exact-sequence}
  0 \longrightarrow \HA^{0,0}(M) \longrightarrow \ker idQ_{1,1}^d\db\dbbar \overset{i\db\dbbar}{\longrightarrow} \mHQ^{1,1} \longrightarrow \HBC^{1,1}(M)
 \end{equation}
 is exact.
\end{proposition}

\begin{proof}
 This follows immediately from the definitions of $\HA^{0,0}(M)$ and $\HBC^{1,1}(M)$, with the understanding that $\HA^{0,0}(M) \to \ker idQ_{1,1}^d\db\dbbar$ is the inclusion map.
\end{proof}

We conclude this section by discussing the real part of the operator $\mD$.

\begin{definition}
 Let $(M^5,T^{1,0},\theta)$ be a pseudohermitian five-manifold.
 We define $R_{1,1}^d \colon \sR^{1,1} \to \sR^{3,1} \oplus \sR^{2,2}$ by
 \begin{equation*}
  R_{1,1}^d := \omD + \mD .
 \end{equation*}
\end{definition}

Note that
\begin{equation}
 \label{eqn:mR}
 R_{1,1}^d = d\hodge \left( \db\db^\ast + \dbbar\dbbar^\ast + P \rwedge \right) + 2\ell^\theta \rwedge .
\end{equation}
In particular, if $\omega$ is $d$-closed, then $dR_{1,1}^d\omega=0$ and $[R_{1,1}^d\omega] \in H^4(M;\bR)$ is proportional to $[\omega] \cup c_1(T^{1,0})$.
In fact, $R_{1,1}^d\omega$ determines a CR invariant class in
\begin{equation*}
 \HBC^{2,2}(M) := \frac{ \ker \left( d \colon \sR^{3,1} \oplus \sR^{2,2} \to \sR^{3,2} \right) }{ \im \left( -i\dbbar\db + i\db\dbbar \colon \sR^{1,1} \to \sR^{3,1} \oplus \sR^{2,2} \right) } .
\end{equation*}
This follows from the transformation law for $R_{1,1}^d$ on $d$-closed $(1,1)$-forms.

\begin{lemma}
 \label{mR-transformation-on-closed}
 Let $(M^5,T^{1,0},\theta)$ be a pseudohermitian five-manifold.
 Then
 \begin{equation*}
  (R_{1,1}^d)^{e^\Upsilon\theta}\omega = (R_{1,1}^d)^\theta\omega + (i\db\dbbar - i\dbbar\db)\left( \Upsilon \rwedge \omega \right)
 \end{equation*}
 for all $\omega \in \ker d \cap \sR^{1,1}$ and all $\Upsilon \in C^\infty(M)$.
\end{lemma}

\begin{proof}
 This follows immediately from \cref{fourth-order-transformation}.
\end{proof}

\section{The $Q$-curvature operator on $\dbbar$-closed $(0,1)$-forms}
\label{sec:kr}

In this section we introduce the $Q$-curvature operator on $\dbbar$-closed $(0,1)$-forms and establish some of its basic properties.
We begin by defining the GJMS operator on $(0,1)$-forms.

\begin{definition}
 Let $(M^5,T^{1,0})$ be a CR five-manifold.
 The \defn{GJMS operator on $(0,1)$-forms}, $L_{0,1}^{\dbbar} \colon \sR^{0,1} \to \sR^{2,2}$, is
 \begin{equation*}
  L_{0,1}^{\dbbar} := i\db\dbbar\db .
 \end{equation*}
\end{definition}

Note that, since~\eqref{eqn:bigraded-rumin-complex} is a bigraded complex, $\ker \dbbar \subset \ker L_{0,1}^{\dbbar}$.

For comparison, the CR Paneitz operator~\cites{GrahamLee1988,Hirachi1990} in dimension three is, after composition with the Hodge star operator, equal to $-\db\dbbar\db$ (cf.\ \cite{CaseYang2020}).
Like the CR Paneitz operator, the GJMS operator on $(0,1)$-forms is a formally self-adjoint CR invariant operator.
Only formal self-adjointness requires a proof.

\begin{lemma}
 \label{L01-formally-self-adjoint}
 Let $(M^5,T^{1,0})$ be a closed CR five-manifold.
 Then the bilinear form $\kL_{0,1} \colon \sR^{0,1} \times \sR^{0,1} \to \bC$,
 \begin{equation*}
  \kL_{0,1}(\omega,\tau) := \int_M \omega \rwedge \overline{L_{0,1}^{\dbbar}\tau} ,
 \end{equation*}
 is Hermitian.
\end{lemma}

\begin{proof}
 Let $\omega,\tau \in \sR^{0,1}$.
 Stokes' Theorem implies that
 \begin{equation*}
  \kL_{0,1}(\omega,\tau) = -i\int_M \omega \rwedge \dbbar\db\dbbar\otau = -i \int_M \dbbar\omega \rwedge \db\dbbar\otau = i\int_M \db\dbbar\omega \rwedge \dbbar\otau .
 \end{equation*}
 Since~\eqref{eqn:bigraded-rumin-complex} is a complex, it holds that $L_{0,1}^{\dbbar}=i\dbbar\db\dbbar$.
 Therefore
 \begin{equation*}
  \overline{\kL_{0,1}(\tau,\omega)} = i\int_M \otau \rwedge \dbbar\db\dbbar\omega = i\int_M \dbbar\otau \rwedge \db\dbbar\omega = \kL_{0,1}(\omega,\tau) . \qedhere
 \end{equation*}
\end{proof}

We now define the $Q$-curvature operator on $\dbbar$-closed $(0,1)$-forms.
This operator is obtained by restriction from an operator on $\sR^{0,1}$.

\begin{definition}
 Let $(M^5,T^{1,0},\theta)$ be a pseudohermitian five-manifold.
 We define $Q_{0,1}^{\dbbar} \colon \sR^{0,1} \to \sR^{2,2}$ by
 \begin{equation*}
  Q_{0,1}^{\dbbar}\omega := \mD\db\omega - \db \left( \omega \krwedge \ell^\theta \right)
 \end{equation*}
 for all $\omega\in\sR^{0,1}$.
\end{definition}

The transformation formula for $Q_{0,1}^{\dbbar}$ is linear in the conformal factor.

\begin{lemma}
 \label{Q01-transformation}
 Let $(M^5,T^{1,0},\theta)$ be a pseudohermitian five-manifold.
 Then
 \begin{multline*}
  (Q_{0,1}^{\dbbar})^{e^\Upsilon\theta}\omega = (Q_{0,1}^{\dbbar})^\theta\omega + i\db\dbbar\db(\Upsilon \krwedge \omega) \\ + i\db(\db\Upsilon \krwedge \dbbar\omega) + i\db(\Upsilon \krwedge \dhor\dbbar\omega) - i\dbbar\Upsilon \krwedge \db\dbbar\omega + i\db\Upsilon \krwedge \dhor\dbbar\omega
 \end{multline*}
 for all $\omega\in\sR^{0,1}$ and all $\Upsilon \in C^\infty(M)$.
\end{lemma}

\begin{proof}
 Combining \cref{lee-form-transformation} with~\eqref{eqn:mD-with-krwedge} yields
 \begin{multline*}
  (Q_{0,1}^{\dbbar})^{e^\Upsilon\theta}\omega = (Q_{0,1}^{\dbbar})^\theta\omega + i\db\dbbar(\Upsilon \krwedge \db\omega) + i\db(\Upsilon \krwedge \dhor\dbbar\omega) \\
   - i\dbbar\Upsilon \krwedge \db\dbbar\omega + i\db\Upsilon \krwedge \dhor\dbbar\omega + i\db(\omega \krwedge \dbbar\db\Upsilon) .
 \end{multline*}
 Type considerations imply that
 \begin{equation*}
  \Upsilon \krwedge \db\omega = \db ( \Upsilon \krwedge \omega ) - \db\Upsilon \krwedge \omega .
 \end{equation*}
 Therefore
 \begin{align*}
  \db\dbbar(\Upsilon \krwedge \db\omega) & = \db\dbbar\db(\Upsilon \krwedge\omega) - \db\dbbar(\db\Upsilon \krwedge \omega) \\
  & = \db\dbbar\db(\Upsilon \krwedge \omega) - \db(\dbbar\db\Upsilon \krwedge \omega ) + \db(\db\Upsilon \krwedge \dbbar\omega) .
 \end{align*}
 The conclusion readily follows.
\end{proof}

Restricting \cref{Q01-transformation} to $\dbbar$-closed $(0,1)$-forms yields \cref{Q01-definition-and-properties}.

\begin{proof}[Proof of \cref{Q01-definition-and-properties}]
 This follows immediately from \cref{Q01-transformation}.
\end{proof}

\Cref{Q01-definition-and-properties} motivates the following definition:

\begin{definition}
 Let $(M^5,T^{1,0},\theta)$ be a pseudohermitian five-manifold.
 The \defn{$Q$-curvature operator on $\dbbar$-closed $(0,1)$-forms} is the restriction
 \begin{equation*}
  Q_{0,1}^{\dbbar} \colon \ker \dbbar \cap \sR^{0,1} \to \sR^{2,2} .
 \end{equation*}
\end{definition}

The $Q$-curvature operator on $\dbbar$-closed $(0,1)$-forms determines a CR invariant pairing of such forms.

\begin{corollary}
 \label{Q01-qua-Q-curvature-operator-invariant-form}
 Let $(M^5,T^{1,0},\theta)$ be a closed pseudohermitian five-manifold.
 Then
 \begin{equation}
  \label{eqn:Q01-bilinear-form}
  \int_M \omega \rwedge \overline{(Q_{0,1}^{\dbbar})^{e^\Upsilon\theta}\tau} = \int_M \omega \rwedge \overline{(Q_{0,1}^{\dbbar})^\theta\tau}
 \end{equation}
 for all $\omega \in \sR^{0,1} \cap \ker \dbbar$ and all $\Upsilon \in C^\infty(M)$.
\end{corollary}

\begin{proof}
 This follows from \cref{Q01-definition-and-properties}, \cref{L01-formally-self-adjoint}, and the fact $\ker\dbbar \subset \ker L_{0,1}^{\dbbar}$.
\end{proof}

We give a name to the CR invariant bilinear form appearing in~\eqref{eqn:Q01-bilinear-form}.

\begin{definition}
 Let $(M^5,T^{1,0})$ be a closed CR five-manifold.
 The \defn{$Q$-form on $\dbbar$-closed $(0,1)$-forms} is given by
 \begin{equation*}
  \kQ_{0,1}(\omega,\tau) := \int_M \omega \rwedge \overline{Q_{0,1}^{\dbbar}\tau}
 \end{equation*}
 for all $\omega , \tau \in \ker \dbbar \cap \sR^{0,1}$.
\end{definition}

Importantly, the $Q$-form on $\dbbar$-closed $(0,1)$-forms is Hermitian.

\begin{lemma}
 \label{Q01-symmetric}
 Let $(M^5,T^{1,0})$ be a closed CR five-manifold.
 Then
 \begin{equation*}
  \kQ_{0,1}(\omega,\tau) = \overline{\kQ_{0,1}(\tau,\omega)}
 \end{equation*}
 for all $\omega,\tau \in \ker \dbbar \cap \sR^{0,1}$.
\end{lemma}

\begin{proof}
 Let $\theta$ be a contact form on $(M^5,T^{1,0})$ and let $\omega,\tau \in \ker\dbbar \cap \sR^{0,1}$.
 Since $\dbbar\omega=0$, Stokes' Theorem implies that
 \begin{equation*}
  \int_M \omega \rwedge \overline{\db(\tau \krwedge \ell^\theta)} = 0 .
 \end{equation*}
 Therefore
 \begin{equation}
  \label{eqn:kQ-evaluate}
  \kQ_{0,1}(\omega,\tau) = \llp \db\omega, (\db\db^\ast + \dbbar\dbbar^\ast + P\rwedge)\db\tau \rrp + 2\int_M \omega \rwedge \pi^{3,1}(\ell^\theta \rwedge \dbbar\otau) .
 \end{equation}
 On the one hand, since $\llp \cdot , \cdot \rrp$ is Hermitian and $P$ is real-valued, it holds that
 \begin{equation}
  \label{eqn:Q01-symmetric-a}
  \llp \db\omega, (\db\db^\ast + \dbbar\dbbar^\ast + P\rwedge)\db\tau \rrp = \overline{ \llp \db\tau, (\db\db^\ast + \dbbar\dbbar^\ast + P\rwedge)\db\omega \rrp} .
 \end{equation}
 On the other hand, since $\omega\in\sR^{0,1}$ and $\dbbar\tau=0$, it holds that
 \begin{equation*}
  \int_M \omega \rwedge \pi^{3,1}(\ell^\theta \rwedge \dbbar\otau) = \int_M \omega \rwedge \left(\ell^\theta \rwedge d\otau \right) .
 \end{equation*}
 Let $\xi \in C^\infty(M;\bC)$ solve $\theta \wedge \ell^\theta \wedge d\otau = \xi\,\theta \wedge d\theta^2$, so that $\ell^\theta \rwedge d\otau = \ell^\theta \wedge d\otau - d(\xi\,\theta \wedge d\theta)$.
 Since $d\ell^\theta=0$ and $\omega \in \sR^{0,1}$ --- and hence $\theta \wedge d\omega \wedge d\theta = 0$ --- Stokes' Theorem yields
 \begin{equation}
  \label{eqn:Q01-symmetric-b}
  \int_M \omega \rwedge \left( \ell^\theta \rwedge d\otau \right) = \int_M \omega \wedge \ell^\theta \wedge d\otau = \int_M \otau \wedge \ell^\theta \wedge d\omega .
 \end{equation}
 The conclusion readily follows from~\eqref{eqn:kQ-evaluate}, \eqref{eqn:Q01-symmetric-a}, and~\eqref{eqn:Q01-symmetric-b}.
\end{proof}

Equation~\eqref{eqn:kQ-evaluate} implies that on the Heisenberg group $\bH^2 \cong \bC^2 \times \bR$ with its standard pseudohermitian structure, 
\begin{equation}
 \label{eqn:heisenberg-q01-form}
 \kQ_{0,1}(\omega,\tau) = \llp \db\omega, (\db\db^\ast + \dbbar\dbbar^\ast)\db\tau \rrp 
\end{equation}
for all compactly-supported $\omega,\tau \in \ker\dbbar \cap \sR^{0,1}$.
In particular, $\kQ_{0,1}(\omega,\omega)\geq0$ for all compactly-supported $\omega \in \ker\dbbar \cap \sR^{0,1}$, with equality if and only if $d\omega=0$.

\Cref{Q01-definition-and-properties} implies that if $\omega\in\sR^{0,1}$ is $\dbbar$-closed, then the condition $\db Q_{0,1}^{\dbbar}\omega=0$ is CR invariant.
Following Branson and Gover~\cite{BransonGover2005}, we introduce the CR invariant space of CR $\dbbar$-harmonic $(0,1)$-forms.

\begin{definition}
 Let $(M^5,T^{1,0})$ be a CR five-manifold.
 The space of \defn{CR $\dbbar$-harmonic $(0,1)$-forms} is
 \begin{equation*}
  \HQ^{0,1} := \left\{ \omega \in \sR^{0,1}(M) \suchthatcolon \dbbar\omega = 0, \db Q_{0,1}^{\dbbar}\omega = 0 \right\} .
 \end{equation*}
\end{definition}

It is natural to ask whether the morphism $\HQ^{0,1} \to H^{0,1}(M)$, $\omega \mapsto [\omega]$, is surjective.
Note that on closed, strictly pseudoconvex CR five-manifolds, $H^{0,1}(M)$ is finite-dimensional, while the fact $\dbbar\bigl( \ker \db \cap \sR^{0,0} \bigr) \subset \HQ^{0,1}$ implies that $\HQ^{0,1}$ is infinite-dimensional.

The kernel of the morphism $\HQ^{0,1} \to H^{0,1}(M)$ can be understood by means of the following long exact sequence (cf.\ \cite{BransonGover2005}*{Proposition~2.5}).

\begin{proposition}
 \label{H01-exact-sequence}
 Let $(M^5,T^{1,0})$ be a CR five-manifold.
 Then the sequence
 \begin{equation*}
  0 \longrightarrow H^{0,0}(M) \longrightarrow \ker \db Q_{0,1}^{\dbbar}\dbbar \overset{\dbbar}{\longrightarrow} \HQ^{0,1} \longrightarrow H^{0,1}(M)
 \end{equation*}
 is exact.
\end{proposition}

\begin{proof}
 This follows immediately from the definitions of $H^{0,q}(M)$, $q\in\{0,1\}$, with the understanding that $H^{0,0}(M) \to \ker \db Q_{0,1}^{\dbbar} \dbbar$ is the inclusion map.
\end{proof}

There is a simple relationship between our $Q$-curvature operators.

\begin{lemma}
 \label{related-q-operators}
 Let $(M^5,T^{1,0})$ be a CR five-manifold.
 If $\omega \in \ker\dbbar \cap \sR^{0,1}$, then $\db\omega$ is $d$-closed and
 \begin{equation*}
  \db Q_{0,1}^{\dbbar} \omega = -\frac{i}{2}dQ_{1,1}^d \db\omega .
 \end{equation*}
\end{lemma}

\begin{proof}
 Let $\omega \in \ker\dbbar\cap \sR^{0,1}$.
 Then $d\db\omega = 0$.
 Equation~\eqref{eqn:mR} then implies that
 \begin{equation*}
  d(\mD + \omD)\db\omega = 0 .
 \end{equation*}
 The conclusion readily follows.
\end{proof}

\Cref{related-q-operators} implies that $\db(\HQ^{0,1}) \subset \mHQ^{1,1}$.
This gives rise to a commutative diagram involving the spaces $H^{0,1}(M)$ and $\HBC^{1,1}(M)$.

\begin{proposition}
 \label{hodge-diagram-commutes}
 Let $(M^5,T^{1,0})$ be a CR five-manifold.
 Then the diagram
 \begin{equation*}
  \begin{tikzpicture}
  	\node (kr-hodge) at (0,2) {$\HQ^{0,1}$};
  	\node (bc-hodge) at (4,2) {$\mHQ^{1,1}$};
  	\node (kr) at (0,0) {$H^{0,1}(M)$};
  	\node (bc) at (4,0) {$\HBC^{1,1}(M)$};
  	% Arrows
  	\draw [->] (kr-hodge) to node[above] {\tiny $\db$} (bc-hodge);
  	\draw [->] (kr) to node[above] {\tiny $\db$} (bc);
  	\draw [->] (kr-hodge) to (kr);
  	\draw [->] (bc-hodge) to (bc);
  \end{tikzpicture}
 \end{equation*}
 commutes.
\end{proposition}

\begin{proof}
 Since $[\db\dbbar f]=0$ in $\HBC^{1,1}(M)$, the morphism $\db \colon H^{0,1}(M) \to \HBC^{1,1}(M)$ is well-defined (cf.\ \cite{Case2021rumin}*{Definition~7.1}).
 The conclusion follows from \cref{related-q-operators}.
\end{proof}

Taking imaginary parts yields a diagram involving the space $H^1(M;\sP)$.

\begin{corollary}
 \label{hodge-diagram-commutes-pluriharmonic}
 Let $(M^5,T^{1,0})$ be a CR five-manifold.
 Then the diagram
 \begin{equation*}
  \begin{tikzpicture}
  	\node (kr-hodge) at (0,2) {$\HQ^{0,1}$};
  	\node (bc-hodge) at (4,2) {$\Real\mHQ^{1,1}$};
  	\node (kr) at (0,0) {$H^{0,1}(M)$};
  	\node (bc) at (4,0) {$H^{1}(M;\sP)$};
  	% Arrows
  	\draw [->] (kr-hodge) to node[above] {\tiny $-\Imaginary\db$} (bc-hodge);
  	\draw [->] (kr) to node[above] {\tiny $-\Imaginary\db$} (bc);
  	\draw [->] (kr-hodge) to (kr);
  	\draw [->] (bc-hodge) to (bc);
  \end{tikzpicture}
 \end{equation*}
 commutes.
\end{corollary}

\begin{proof}
 This follows immediately from \cref{hodge-diagram-commutes}.
\end{proof}

\section{The scalar $Q$-curvature}
\label{sec:scalar}

In this section we use our $Q$-curvature operators to study the (critical) scalar $Q$-curvature~\cite{FeffermanHirachi2003}.
We first define the scalar $Q$-curvature and the scalar GJMS operator in terms of the $Q$-curvature operator on $d$-closed $(1,1)$-forms.

\begin{definition}
 \label{defn:scalar}
 Let $(M^5,T^{1,0},\theta)$ be a pseudohermitian five-manifold.
 The \defn{scalar GJMS operator} $\Lscal \colon \sR^{0,0} \to \sR^{3,2}$ and the \defn{scalar $Q$-curvature} $\Qscal \in \sR^{3,2}$ are
 \begin{equation}
  \label{eqn:scalar}
  \begin{split}
   \Lscal & := id Q_{1,1}^d \db\dbbar , \\
   \Qscal & := d Q_{1,1}^d \ell^\theta .
  \end{split}
 \end{equation}
\end{definition}

As we will see later in this section, these definitions are equivalent to their respective definitions given in terms of the ambient metric~\cite{FeffermanHirachi2003}.
For now, we justify this definition through three observations.

First, by expressing $\Lscal$ in terms of the $Q$-curvature operator on $\dbbar$-closed $(0,1)$-forms, we find that $\Lscal$ is formally self-adjoint.

\begin{lemma}
 \label{L0-from-Q01}
 Let $(M^5,T^{1,0},\theta)$ be a pseudohermitian five-manifold.  Then
 \begin{equation}
  \label{eqn:L00-from-Q01}
  \Lscal = -2\db Q_{0,1}^{\dbbar} \dbbar .
 \end{equation}
 In particular, if $M$ is closed, then
 \begin{equation*}
  \kL_{0,0}(\varphi, \psi) := \int_M \varphi \rwedge \overline{\Lscal\psi}
 \end{equation*}
 defines an Hermitian form on $\sR^{0,0}$.
\end{lemma}

\begin{proof}
 Equation~\eqref{eqn:L00-from-Q01} follows immediately from \cref{related-q-operators}.
 
 Let $M$ be closed.  Stokes' Theorem and~\eqref{eqn:L00-from-Q01} imply that
 \begin{equation}
  \label{eqn:kL00-from-Q-form}
  \kL_{0,0}(\varphi , \psi) = 2\kQ_{0,1}(\dbbar\varphi, \dbbar\psi)
 \end{equation}
 for all $\varphi,\psi \in \sR^{0,0}$.
 The conclusion follows from \cref{Q01-symmetric}.
\end{proof}

Second, \eqref{eqn:heisenberg-q01-form} implies that $\hodge \Lscal$ is a nonnegative sixth-order operator on the Heisenberg group $\bH^2$ with its standard pseudohermitian structure.

Third, $\Lscal$ is CR invariant and equals the conformal linearization of $\Qscal$.

\begin{lemma}
 \label{gjms-and-q-transformation}
 Let $(M^5,T^{1,0},\theta)$ be a pseudohermitian five-manifold.
 Then
 \begin{align}
  \label{eqn:gjms-transformation} (\Lscal)^{e^\Upsilon\theta} f & = (\Lscal)^\theta f , \\
  \label{eqn:q-transformation} (\Qscal)^{e^\Upsilon\theta} & = (\Qscal)^\theta + (\Lscal)^\theta\Upsilon
 \end{align}
 for all $f \in \sR^{0,0}$ and all $\Upsilon\in C^\infty(M)$.
 In particular, if $M$ is closed, then
 \begin{equation*}
  \int_M f\,(\Qscal)^{e^\Upsilon\theta} = \int_M f\,(\Qscal)^\theta
 \end{equation*}
 for all $f \in \ker \Lscal$.
\end{lemma}

\begin{proof}
 \Cref{mQ01-definition-and-properties} yields~\eqref{eqn:gjms-transformation}.
 \Cref{mQ01-definition-and-properties,lee-form-transformation} yield~\eqref{eqn:q-transformation}.
 The final conclusion follows from \cref{L0-from-Q01} and~\eqref{eqn:q-transformation}.
\end{proof}

Recall that $(M^5,T^{1,0},\theta)$ is \defn{$Q$-flat} if $(\Qscal)^\theta=0$.
\Cref{scalar-q-flat-characterization} is a consequence of \cref{gjms-and-q-transformation}.

\begin{proof}[Proof of \cref{scalar-q-flat-characterization}]
 Let $\theta$ be a contact form on $(M^5,T^{1,0})$.
 If $\theta$ is $Q$-flat, then $\ell^\theta\in\mHQ^{1,1}$.
 Conversely, if $[\ell^\theta] \in \im \bigl( \mHQ^{1,1} \to \HBC^{1,1}(M) \bigr)$, then there is an $\omega\in\mHQ^{1,1}$ and an $f\in\sR^{0,0}$ such that
 \begin{equation*}
  \ell^\theta = \omega + i\db\dbbar f .
 \end{equation*}
 Applying $d(Q_{1,1}^d)^\theta$ and taking real parts yields $(\Qscal)^\theta = (\Lscal)^\theta\Real f$.
 We conclude from \cref{gjms-and-q-transformation} that $e^{-\Real f}\theta$ is $Q$-flat.
\end{proof}

Our definition of the scalar $Q$-curvature yields an easy computation of the pairing of a CR pluriharmonic function and $\Qscal$.
More generally:

\begin{theorem}
 \label{q-functional}
 Let $(M^5,T^{1,0},\theta)$ be a closed pseudohermitian five-manifold.
 Then
 \begin{equation}
  \label{eqn:q-functional}
  \int_M f \rwedge \left(\Qscal\right)^\theta = \int_M dd_b^cf \rwedge \hodge (\db\db^\ast + \dbbar\dbbar^\ast + P\rwedge)\ell^\theta + 2\int_M d_b^cf \rwedge \ell^\theta \rwedge \ell^\theta
 \end{equation}
 for all $f\in \sR^{0,0}$.
 In particular, if
 \begin{enumerate}
  \item $H^{1}(M;\bR)=0$,
  \item $c_1^2(T^{1,0})=0$ in $H^4(M;\bR)$, or
  \item $(M^5,T^{1,0})$ is strictly pseudoconvex,
 \end{enumerate}
 then $\Qscal \in \sP^\perp$.
\end{theorem}

\begin{proof}
 Let $f\in \sR^{0,0}$.
 Using Stokes' Theorem, we deduce that
 \begin{equation*}
  \int_M f \rwedge \left(\Qscal\right)^\theta = \int_M f \rwedge d(Q_{1,1}^d)^\theta\ell^\theta = -\int_M df \rwedge (Q_{1,1}^d)^\theta\ell^\theta = \int_M d_b^cf \rwedge (R_{1,1}^d)^\theta\ell^\theta .
 \end{equation*}
 Combining these displays with the definition of $R_{1,1}^d$ yields~\eqref{eqn:q-functional}.
 
 Suppose now that $u\in\sP$.
 Then $dd_b^cu=0$.
 In particular, \eqref{eqn:q-functional} yields
 \begin{equation}
  \label{eqn:Q00-sP-pairing}
  \int_M u \rwedge \Qscal = 2 \left\lp [d_b^cu] \cup [\ell^\theta] \cup [\ell^\theta] , [M] \right\rp ,
 \end{equation}
 where $[d_b^cu] \in H^1(M;\bR)$ and $[\ell^\theta] \in H^2(M;\bR)$.
 It immediately follows that if $H^1(M;\bR)=0$, then $\Qscal \in \sP^\perp$.
 Equation~\eqref{eqn:lee-class-in-H2} implies that if $c_1^2(T^{1,0})=0$, then $\Qscal \in \sP^\perp$.
 Finally, if $(M,T^{1,0})$ is strictly pseudoconvex, then a result of Takeuchi~\cite{Takeuchi2018}*{Theorem~1.1} implies that $c_1^2(T^{1,0})=0$.
\end{proof}

The proof of \cref{q-functional} includes the key ingredient in the proof of \cref{q00-integral}:

\begin{proof}[Proof of \cref{q00-integral}]
 Equations~\eqref{eqn:lee-class-in-H2} and~\eqref{eqn:Q00-sP-pairing} imply that~\eqref{eqn:q00-integral} holds.
 The remaining conclusion follows from \cref{q-functional}.
\end{proof}

Another consequence of \cref{q-functional}, is that on closed, strictly pseudoconvex five-manifolds, the condition that the GJMS operator have trivial kernel is sufficient for the existence of a $Q$-flat contact form.

\begin{corollary}
 \label{Q0-flat-from-trivial-kernel}
 Every $(M^5,T^{1,0})$ be a closed, strictly pseudoconvex five-manifold for which $\ker \Lscal = \sP$ admits a $Q$-flat contact form.
\end{corollary}

\begin{proof}
 Since $\ker \Lscal = \sP$, we conclude from \cref{q-functional} that $\Qscal$ annihilates $\ker \Qscal$.
 The conclusion follows from a result of Takeuchi~\cite{Takeuchi2020b}*{Proposition~1.4}.
\end{proof}

Note that, unlike in dimension three~\cite{Takeuchi2019}, there are closed, strictly pseudoconvex five-manifolds for which $\Lscal$ has nontrivial kernel~\cite{Takeuchi2017}.

We conclude this section by proving that $\Qscal$ and $\Lscal$ are equivalent to the scalar $Q$-curvature $Q_{\mathrm{FH}}$ and the scalar GJMS operator $L_{\mathrm{FH}}$, respectively, of Fefferman and Hirachi~\cite{FeffermanHirachi2003}.
This is accomplished using the explicit formula for $Q_{\mathrm{FH}}$ computed by Case and Gover~\cite{CaseGover2013} via the CR tractor calculus~\cite{GoverGraham2005}.

\begin{proposition}
 \label{Q0-is-standard}
 Let $(M^5,T^{1,0},\theta)$ be a pseudohermitian five-manifold.
 Then $Q_{\mathrm{FH}} = 8\hodge \Qscal$.
\end{proposition}

\begin{proof}
 Case and Gover proved~\cite{CaseGover2013}*{Proposition~8.4} that
 \begin{equation*}
  \frac{1}{16}Z_{\bar E}Z_{A}Z_{\bar B} Q_{\mathrm{FH}} = \bD_{\bar E}K_{A\bar B} - Z_{\bar E}S_{A\bar BC\bar D}K^{\bar DC} ,
 \end{equation*}
 where
 \begin{multline*}
  K_{A\bar B} := -W_A{}^\alpha W_{\bar B}{}^{\bar\beta}E_{\alpha\bar\beta} + W_A{}^\alpha Z_{\bar B}\nabla^{\bar\nu}E_{\alpha\bar\nu} \\ 
   + Z_AW_{\bar B}{}^{\bar\beta}\nabla^\mu E_{\mu\bar\beta} - Z_AZ_{\bar B}\left( E_{\mu\bar\nu}E^{\bar\nu\mu} + \nabla^\mu\nabla^{\bar\nu}E_{\mu\bar\nu} \right) ,
 \end{multline*}
 $\bD_{\bar E}$ is the CR tractor $D$-operator~\cite{CaseGover2013}*{pg.\ 581} on densities of weight $(-1,-1)$, $S_{A\bar BC\bar D}$ is the CR Weyl tractor~\cite{CaseGover2013}*{Proposition~8.1}, and $Z_A$ and $W_A{}^\alpha$ are the contact form-dependent bundle injections~\cite{CaseGover2013}*{Equation~(3.4)}.
 It follows immediately from these definitions that
 \begin{align*}
  S_{A\bar BC\bar D}K^{\bar DC} & \equiv -Z_AZ_{\bar B}U_{\mu\bar\nu}E^{\bar\nu\mu} \mod W_A{}^\alpha , W_{\bar B}{}^{\bar\beta} , \\
  \bD_{\bar E}K_{A\bar B} & = -Z_{\bar E}(\nabla^{\bar\nu}\nabla_{\bar\nu} K_{A\bar B} + i\nabla_0 K_{A\bar B} - PK_{A\bar B}) ,
 \end{align*}
 where $\nabla$ denotes the CR tractor connection~\cite{CaseGover2013}*{Equation~(3.5)},
 \begin{equation*}
  U_{\alpha\bar\beta} = -\Real\nabla_{\alpha}\nabla_{\bar\beta}P + 2\Real\nabla_\alpha\nabla^\mu E_{\mu\bar\beta} + P_\alpha{}^\mu P_{\mu\bar\beta} - A_\alpha{}^{\bar\nu}A_{\bar\nu\bar\beta} + uh_{\alpha\bar\beta} ,
 \end{equation*}
 and $u$ is such that $U_\mu{}^\mu=0$ (cf. \cite{CaseGover2013}*{pg.\ 575}).
 Since $\dim_{\bC}T^{1,0}=2$, it holds that
 \begin{equation*}
  P_\alpha{}^\mu P_{\mu\bar\beta} - \frac{1}{2}P^{\bar\nu\mu}P_{\mu\bar\nu}h_{\alpha\bar\beta} = PE_{\alpha\bar\beta} .
 \end{equation*}
 A straightforward computation yields
 \begin{multline*}
  \frac{1}{8}Q_{\mathrm{FH}} = 2\Real \nabla^\gamma \bigl( \nabla_\gamma(E_{\mu\bar\nu}E^{\bar\nu\mu}) + 4E_\gamma{}^\mu\nabla^{\bar\nu}E_{\mu\bar\nu} + \nabla_\gamma\nabla^\mu\nabla^{\bar\nu}E_{\mu\bar\nu} \\
   + 2iA_{\gamma}{}^{\bar\nu}\nabla^{\mu}E_{\mu\bar\nu} - \nabla_\mu(PE_\gamma{}^\mu) \bigr) .
 \end{multline*}
 Commutator identities~\cite{Lee1988}*{Lemma~2.3} yield $\nabla^\mu\nabla^{\bar\nu}E_{\mu\bar\nu} = \nabla^{\bar\nu}\nabla^\mu E_{\mu\bar\nu}$ and
 \begin{align*}
  \nabla^\gamma\left( \nabla_\gamma\nabla^\mu\nabla^{\bar\nu}E_{\mu\bar\nu} + 2iA_{\gamma}{}^{\bar\nu}\nabla^{\mu}E_{\mu\bar\nu} \right) & = \nabla^\mu\nabla^{\bar\nu} \left( 2\nabla_\mu\nabla^\gamma E_{\gamma\bar\nu} - h_{\mu\bar\nu}\nabla^\gamma\nabla^{\bar\sigma}E_{\gamma\bar\sigma} \right) \\
   & = \nabla^{\bar\nu}\nabla^\mu \left( 2\nabla_\mu\nabla^\gamma E_{\gamma\bar\nu} - h_{\mu\bar\nu}\nabla^\gamma\nabla^{\bar\sigma}E_{\gamma\bar\sigma} \right) .
 \end{align*}
 Therefore
 \begin{multline*}
  \frac{1}{8}Q_{\mathrm{FH}} = 2\Real \nabla^\mu \nabla^{\bar\nu} \left( \nabla_\mu \nabla^\gamma E_{\gamma\bar\nu} + \nabla_{\bar\nu}\nabla^{\bar\sigma} E_{\mu\bar\sigma} - h_{\mu\bar\nu}\nabla^\gamma\nabla^{\bar\sigma} E_{\gamma\bar\sigma} - PE_{\mu\bar\nu} \right) \\
   + \Real \nabla^\gamma \left( \nabla_\gamma (E_{\mu\bar\nu}E^{\bar\nu\mu}) + 4E_\gamma{}^\mu\nabla^{\bar\nu}E_{\mu\bar\nu} \right) .
 \end{multline*}
 The conclusion readily follows from~\eqref{eqn:db-high}, \eqref{eqn:lee-form}, \eqref{eqn:mD-frame}, and~\eqref{eqn:scalar}.
\end{proof}

It follows that $\Lscal$ is equivalent to $L_{\mathrm{FH}}$.

\begin{proposition}
 \label{L0-is-standard}
 Let $(M^5,T^{1,0},\theta)$ be a pseudohermitian five-manifold.
 Then $L_{\mathrm{FH}} = 8\hodge \Lscal$.
\end{proposition}

\begin{proof}
 Let $\Upsilon \in C^\infty(M)$.
 Equation~\eqref{eqn:q-transformation} implies that
 \begin{equation*}
  e^{3\Upsilon}\hodge^{e^\Upsilon\theta} (\Qscal)^{e^\Upsilon\theta} = \hodge^\theta (\Qscal)^\theta + \hodge^\theta (\Lscal)^\theta \Upsilon .
 \end{equation*}
 Recall that~\cite{FeffermanHirachi2003}*{Equation~(3.3)}
 \begin{equation*}
  e^{3\Upsilon} Q_{\mathrm{FH}}^{e^\Upsilon\theta} = Q_{\mathrm{FH}}^\theta + L_{\mathrm{FH}}^\theta \Upsilon .
 \end{equation*}
 The conclusion follows from \cref{Q0-is-standard}.
\end{proof}

We conclude this section by proving \cref{recover-scalar}.

\begin{proof}[Proof of \cref{recover-scalar}]
 This follows immediately from \cref{Q0-is-standard,L0-is-standard}.
\end{proof}

\section{Towards a Hodge theory for $\HQ^{0,1}$}
\label{sec:hodge}

In this section we characterize those closed, strictly pseudoconvex CR manifolds for which $\HQ^{0,1} \to H^{0,1}(M)$ is surjective.
We accomplish this by introducing a CR invariant subspace $\HB^{0,1} \subset \HQ^{0,1}$ such that $\dim \HQ^{0,1}/\HB^{0,1} = \dim H^{0,1}(M)$.

\begin{definition}
 Let $(M^5,T^{1,0})$ be a CR five-manifold.
 We set
 \begin{equation*}
  \HB^{0,1} := \left\{ \dbbar f \suchthatcolon f \in \sR^{0,0}, Q_{0,1}^{\dbbar}\dbbar f \in \im \db \right\} .
 \end{equation*}
\end{definition}

\Cref{Q01-definition-and-properties} implies that $\HB^{0,1}$ is a CR invariant vector space.
It is clear that $\HB^{0,1} \subset \HQ^{0,1}$ and that $\dbbar( \ker \db \cap \sR^{0,0}) \subset \HB^{0,1}$.
We deduce from embeddability~\cite{Boutet1975} that $\HB^{0,1}$ is infinite-dimensional on closed, strictly pseudoconvex CR five-manifolds.
Note that the identity map and the $Q$-curvature operator on $\dbbar$-closed $(0,1)$-forms induce well-defined morphisms $\iota \colon \HQ^{0,1} / \HB^{0,1} \to H^{0,1}(M)$ and $Q_{0,1} \colon \HQ^{0,1} / \HB^{0,1} \to H_{\db}^{2,2}(M)$, respectively.

The first step in proving that $\dim \HQ^{0,1}/\HB^{0,1} = \dim H^{0,1}(M)$ on closed, strictly pseudoconvex CR five-manifolds is to construct a short exact sequence involving the above spaces and their ``duals'' $\HQdual_{0,1}$ and $\HBdual_{0,1}$.

\begin{definition}
 Let $(M^5,T^{1,0})$ be a CR five-manifold.
 We set
 \begin{align*}
  \HQdual_{0,1} & := \left\{ \omega \in \sR^{2,2} \suchthatcolon \exists \tau \in \ker\dbbar \cap \sR^{0,1} , \db\omega = \db Q_{0,1}^{\dbbar}\tau \right\} , \\
  \HBdual_{0,1} & := \left\{ \omega \in \sR^{2,2} \suchthatcolon \exists f \in \sR^{0,0} , \omega - Q_{0,1}^{\dbbar}\dbbar f \in \im \db \right\} .
 \end{align*}
\end{definition}

\Cref{Q01-definition-and-properties} implies that both $\HQdual_{0,1}$ and $\HBdual_{0,1}$ are CR invariant vector spaces.
It is clear that $\HBdual_{0,1} \subset \HQdual_{0,1}$.

The aforementioned short exact sequence is as follows (cf.\ \cite{BransonGover2005}*{Equation~(13)}):

\begin{proposition}
 \label{H01-B01-short-exact-sequence}
 Let $(M^5, T^{1,0})$ be a CR five-manifold.
 Then the sequence
 \begin{equation}
  \label{eqn:H01-B01-short-exact-sequence}
  0 \longrightarrow \frac{\HQ^{0,1}}{\HB^{0,1}} \overset{\iota + Q_{0,1}^{\dbbar}}{\longrightarrow} H^{0,1}(M) \oplus H_{\db}^{2,2}(M) \overset{P}{\longrightarrow} \frac{\HQdual_{0,1}}{\HBdual_{0,1}} \longrightarrow 0
 \end{equation}
 is exact, where $P$ is induced by
 \begin{equation*}
  (\sR^{0,1} \cap \ker \dbbar) \oplus (\sR^{2,2} \cap \ker \db) \ni \omega + \tau \mapsto \tau - Q_{0,1}^{\dbbar}\omega\in \HQdual_{0,1} .
 \end{equation*}
 In particular, if $(M^5,T^{1,0})$ is closed and strictly pseudoconvex, then $\HQ^{0,1}/\HB^{0,1}$ and $\HQdual_{0,1}/\HBdual_{0,1}$ are finite-dimensional.
\end{proposition}

\begin{proof}
 Observe that if $f\in\sR^{0,0}$ and $\xi\in\sR^{1,2}$, then
 \begin{equation*}
  P(\dbbar f + \db\xi) = \db\xi - Q_{0,1}^{\dbbar}\dbbar f \in \HBdual_{0,1} .
 \end{equation*}
 Therefore $P$ is well-defined.
 Clearly $P \circ (\iota + Q_{0,1}^{\dbbar}) = 0$.
 
 Let $\omega \in \HQ^{0,1}$ be such that $[\omega] + [Q_{0,1}^{\dbbar}\omega] = 0$ in $H^{0,1}(M) \oplus H_{\db}^{2,2}(M)$.
 Then there are $f\in\sR^{0,0}$ and $\xi\in\sR^{1,2}$ such that $\omega=\dbbar f$ and $Q_{0,1}^{\dbbar}\omega = \db\xi$.
 We conclude that~\eqref{eqn:H01-B01-short-exact-sequence} is exact at $\HQ^{0,1} / \HB^{0,1}$.
 
 Let $\omega \in \ker\dbbar \cap \sR^{0,1}$ and $\tau \in \ker \db \cap \sR^{2,2}$ be such that
 \begin{equation*}
  P([\omega] + [\tau]) = 0 \in \HQdual_{0,1} / \HBdual_{0,1} .
 \end{equation*}
 Then there is a $\zeta \in \HBdual_{0,1}$ such that $\tau - Q_{0,1}^{\dbbar}\omega = \zeta$.
 As an element of $\HBdual_{0,1}$, there are $f\in\sR^{0,0}$ and $\xi\in\sR^{1,2}$ such that $\zeta = Q_{0,1}^{\dbbar}\dbbar f + \db\xi$.
 Therefore
 \begin{equation*}
  \tau - \db\xi = Q_{0,1}^{\dbbar}(\omega + \dbbar f) .
 \end{equation*}
 We conclude that~\eqref{eqn:H01-B01-short-exact-sequence} is exact at $H^{0,1}(M) \oplus H_{\db}^{2,2}(M)$.
 
 Let $\omega \in \HQdual_{0,1}$.
 Then there is a $\tau \in \ker \dbbar \cap \sR^{0,1}$ such that $\db\omega = \db Q_{0,1}^{\dbbar}\tau$.
 Therefore $\omega - Q_{0,1}^{\dbbar}\tau \in \ker \db$.
 In particular, $[\omega] = P([-\tau] + [\omega - Q_{0,1}^{\dbbar}\tau]) \in \HQdual_{0,1} / \HBdual_{0,1}$.
 We conclude that~\eqref{eqn:H01-B01-short-exact-sequence} is exact at $\HQdual_{0,1} / \HBdual_{0,1}$.
 
 The final conclusion follows from \cref{hodge}.
\end{proof}

The second step in proving that $\dim \HQ^{0,1}/\HB^{0,1} = \dim H^{0,1}(M)$ is to show that integration defines a perfect pairing between $\HQ^{0,1}/\HB^{0,1}$ and $\HQdual_{0,1}/\HBdual_{0,1}$.

\begin{theorem}
 \label{mH-to-H-surjective}
 Let $(M^5,T^{1,0})$ be a closed, strictly pseudoconvex five-manifold.
 Then the map $(\omega,\tau) \mapsto \int_M \omega \rwedge \otau$ induces a CR invariant perfect pairing
 \begin{equation*}
  \kI \colon \frac{\HQ^{0,1}}{\HB^{0,1}} \times \frac{\HQdual_{0,1}}{\HBdual_{0,1}} \to \bC .
 \end{equation*}
 In particular, $\HQ^{0,1} \to H^{0,1}(M)$ is surjective if and only if
 \begin{equation*}
  \HB^{0,1} = \ker \bigl( \HQ^{0,1} \to H^{0,1}(M) \bigr) .
 \end{equation*}
\end{theorem}

\begin{proof}
 Let $\dbbar f \in \HB^{0,1}$ and $\tau \in \HQdual_{0,1}$.
 Then there is a $\xi \in \ker \dbbar \cap \sR^{0,1}$ such that $\db\tau = \db Q_{0,1}^{\dbbar}\xi$.
 It follows from Stokes' Theorem that
 \begin{equation*}
  \int_M \dbbar f \rwedge \otau = -\int_M f \rwedge \overline{\db\tau} = -\int_M f \rwedge \overline{\db Q_{0,1}^{\dbbar}\xi} = \kQ_{0,1}(\dbbar f, \xi) .
 \end{equation*}
 Since $\dbbar f \in \HB^{0,1}$, there is a $\zeta \in \sR^{1,2}$ such that $Q_{0,1}^{\dbbar}\dbbar f = \db\zeta$.
 Using Stokes' Theorem and \cref{Q01-symmetric}, we conclude that
 \begin{equation*}
  \int_M \dbbar f \rwedge \otau = \overline{ \kQ_{0,1}(\xi, \dbbar f) } = \int_M \oxi \rwedge \db\zeta = 0 .
 \end{equation*}
 
 Similarly, let $\omega \in \HQ^{0,1}$ and $\tau\in\HBdual_{0,1}$.
 Then there are $f\in\sR^{0,0}$ and $\xi\in\sR^{1,2}$ such that $\tau = \db\xi + Q_{0,1}^{\dbbar}\dbbar f$.
 As $\dbbar\omega=0$, we conclude from Stokes' Theorem and \cref{Q01-symmetric} that
 \begin{equation*}
  \int_M \omega \rwedge \otau = \kQ_{0,1} (\omega, \dbbar f) = \overline{\kQ_{0,1}(\dbbar f, \omega)} .
 \end{equation*}
 Since $\omega\in\HQ^{0,1}$, it holds that $\db Q_{0,1}^{\dbbar}\omega = 0$, and hence $\int\omega\rwedge\otau = 0$.
 We conclude that $\kI$ is well-defined.
 By the final conclusion of \cref{H01-B01-short-exact-sequence}, it suffices to show that $\kI$ is nondegenerate to conclude that it is a perfect pairing.
 To that end, let $\theta$ be a contact form on $(M^5,T^{1,0})$.
 
 Suppose first that $\omega \in \HQ^{0,1}$ is such that $\int \omega \rwedge \otau = 0$ for all $\tau\in\HQdual_{0,1}$.
 Since $\ker\db \cap \sR^{2,2} \subset \HQdual_{0,1}$, we conclude from Serre duality that $[\omega]=0$ in $H^{0,1}(M)$.
 Now let $\eta \in \ker\dbbar \cap \sR^{0,1}$.
 Then $Q_{0,1}^{\dbbar}\eta \in \HQdual_{0,1}$.
 Using \cref{Q01-symmetric}, we compute that
 \begin{equation*}
  0 = \int_M \oomega \rwedge Q_{0,1}^{\dbbar}\eta = \int_M \eta \rwedge \overline{Q_{0,1}^{\dbbar}\omega} .
 \end{equation*}
 Since $\eta$ is arbitrary, we conclude from Serre duality that $[Q_{0,1}^{\dbbar}\omega]=0$ in $H_{\db}^{2,2}(M)$.
 Therefore $\omega \in \HB^{0,1}$, and hence $\kI$ is nondegenerate on the left.
 
 Suppose next that $\omega \in \HQdual_{0,1}$ is such that $\int \tau \rwedge \oomega = 0$ for all $\tau\in\HQ^{0,1}$.
 Since $\dbbar(\ker \Lscal) \subset \HQ^{0,1}$, we deduce that $\hodge\db\omega$ is $L^2$-orthogonal to $\ker \Lscal$.
 Combining \cref{L0-is-standard} with a result of Takeuchi~\cite{Takeuchi2020b}*{Theorem~1.1} implies that there is an $f\in\sR^{0,0}$ such that $\db\omega = \db Q_{0,1}^{\dbbar}\dbbar f$.
 Set $\omega_1 := \omega - Q_{0,1}^{\dbbar}\dbbar f$.
 Then $\db\omega_1=0$.
 
 Set
 \begin{equation*}
  \mT^{0,1} := \left\{ [\tau] \suchthatcolon \tau \in \HQ^{0,1} \right\} \subset H^{0,1}(M) .
 \end{equation*}
 \Cref{Q01-symmetric} and Stokes' Theorem imply that
 \begin{equation*}
  0 = \int_M \tau \rwedge \oomega = \int_M \tau \rwedge \oomega_1
 \end{equation*}
 for all $\tau \in \HQ^{0,1}$ and all $f \in \sR^{0,0}$.
 Therefore
 \begin{equation}
  \label{eqn:perp-mT1}
  \kJ([\tau], [\omega_1]) = 0
 \end{equation}
 for all $[\tau] \in \mT^{0,1}$.
 \Cref{Q01-symmetric,L0-from-Q01} imply that if $[\tau] \in \mT^{0,1}$ and $\psi \in \ker \Lscal$, then
 \begin{equation}
  \label{eqn:perp-mT2}
  \overline{\kJ( [\tau] , [Q_{0,1}^{\dbbar}\dbbar\psi] )} = \kQ_{0,1}(\dbbar\psi, \tau) = -\int_M \psi \rwedge \overline{\db Q_{0,1}^{\dbbar}\tau} = 0 .
 \end{equation}
 
 Now, a result of Takeuchi~\cite{Takeuchi2020b}*{Theorem~1.1} implies that if $\tau \in \sR^{0,1}$, then there are $\varphi \in \Lscal$ and $f \in \sR^{0,0}$ such that
 \begin{equation}
  \label{eqn:takeuchi}
  \db Q_{0,1}^{\dbbar}\tau = \hodge \varphi + \db Q_{0,1}^{\dbbar}\dbbar f .
 \end{equation}
 Let $\{ [\tau_j] \}_{j\in J}$ be a (necessarily finite-dimensional) basis for a complement of $\mT^{0,1}$ in $H^{0,1}(M)$.
 Equation~\eqref{eqn:takeuchi} implies that we may take representatives $\tau_j \in \ker\dbbar \cap \sR^{0,1}$ of $[\tau_j]$ such that $\db Q_{0,1}^{\dbbar}\tau_j = \hodge \varphi_j$, $\varphi_j \in \ker \Lscal$.
 Since $\{ [\tau_j] \}_{j\in J}$ is a complement for $\mT^{0,1} \subset H^{0,1}(M)$, the set $\{ \varphi_j \}$ is linearly independent;
 indeed, if $\sum_j a_j\varphi_j=0$, then $\sum_j a_j\tau_j \in \HQ^{0,1}$.
 We conclude that there are constants $\{ c_k \}_{k \in J}$ such that
 \begin{equation}
  \label{eqn:choose-coefficients}
  \frac{1}{2}\sum_{k\in J}\int_M \varphi_j\overline{c_k\varphi_k}\,\theta \wedge d\theta^2 = \kJ( [\tau_j], [\omega_1] )
 \end{equation}
 for all $j \in J$.
 Set $\psi := \sum_k c_k\varphi_k$ and $\omega_2 := \omega_1 + Q_{0,1}^{\dbbar}\dbbar\psi$.
 On the one hand, \cref{Q01-symmetric} and~\eqref{eqn:choose-coefficients} imply that
 \begin{align*}
  \kJ([\tau_j],[\omega_2]) & = \kJ([\tau_j],[\omega_1]) + \kQ_{0,1}(\tau_j,\dbbar\psi)  = \kJ([\tau_j],[\omega_1]) - \int_M \opsi \rwedge \db Q_{0,1}^{\dbbar}\tau_j = 0
 \end{align*}
 for all $j\in J$.
 On the other hand, \eqref{eqn:perp-mT1} and~\eqref{eqn:perp-mT2} imply that
 \begin{equation*}
  0 = \kJ( [\tau] , [\omega_1] ) = \kJ( [\tau] , [\omega_2] )
 \end{equation*}
 for all $[\tau] \in \mT^{0,1}$.
 We conclude from Serre duality that $[\omega_2]=0$ in $H_{\db}^{2,2}(M)$.
 In particular, $\omega \in \HBdual_{0,1}$, and hence $\kI$ is nondegenerate on the right.
 
 Finally, \cref{hodge} implies that $H^{0,1}(M)$ and $H_{\db}^{2,2}(M)$ are isomorphic as finite-dimensional vector spaces.
 Since $\kI \colon (\HQ^{0,1} / \HB^{0,1}) \oplus (\HQdual_{0,1} / \HBdual_{0,1}) \to \bC$ is a perfect pairing, we conclude from \cref{H01-B01-short-exact-sequence} that
 \begin{equation*}
  \dim \HQ^{0,1} / \HB^{0,1} = \dim H^{0,1}(M) .
 \end{equation*}
 In particular, the canonical morphism $\iota \colon \HQ^{0,1} / \HB^{0,1} \to H^{0,1}(M)$ is injective if and only if it is surjective.
 The final conclusion readily follows.
\end{proof}

\begin{remark}
 Since $\HQ^{0,1}$ is infinite-dimensional, the pair $(\dbbar, \db Q_{0,1}^{\dbbar})$ is not graded injectively hypoelliptic.
 It is for this reason that our proof that $\kI$ is nondegenerate on the right is more complicated than the proof of the analogous result for closed Riemannian manifolds~\cite{BransonGover2005}*{Theorem~2.11}.
\end{remark}

\section{Towards the Weak Lee Conjecture}
\label{sec:strong-lee}

In this section we use the $Q$-form on $\dbbar$-closed $(0,1)$-forms to formulate a condition sufficient for the morphism $\HQ^{0,1} \to H^{0,1}(M)$ to be surjective.
This condition is also sufficient for the validity of the Weak Lee Conjecture.

\Cref{Q01-symmetric} justifies the following definition.

\begin{definition}
 A closed CR five-manifold $(M^5,T^{1,0})$ has \defn{nonnegative $Q$-curv\-ature operator on $\dbbar$-closed $(0,1)$-forms} if $\kQ_{0,1}(\omega,\omega)\geq0$ for all $\omega \in \ker \dbbar \cap \sR^{0,1}$.
\end{definition}

The main result of this section is the following sufficient condition for the surjectivity of the canonical morphism $\HQ^{0,1} \to H^{0,1}(M)$ (cf.\ \cite{AubryGuillarmou2011}*{Proposition~1.4}).

\begin{proposition}
 \label{Q-nonnegative-consequence}
 Let $(M^5,T^{1,0})$ be a closed, strictly pseudoconvex five-manifold with nonnegative $Q$-curvature operator on $\dbbar$-closed $(0,1)$-forms.
 Then the canonical morphism $\HQ^{0,1} \to H^{0,1}(M)$ is surjective.
\end{proposition}

\begin{proof}
 Let $\omega\in\HQ^{0,1}$ be such that $[\omega]=0\in H^{0,1}(M)$.
 \Cref{L0-from-Q01} implies that there is an $f\in\ker \Lscal$ such that $\omega=\dbbar f$.
 Equation~\eqref{eqn:kL00-from-Q-form} then implies that $\kQ_{0,1}(\dbbar f,\dbbar f)=0$.
 Since $\kQ_{0,1}\geq0$, we deduce that $\kJ(\eta, Q_{0,1}^{\dbbar}\dbbar f)=0$ for all $\eta \in \ker\dbbar \cap \sR^{0,1}$.
 Serre duality implies that $Q_{0,1}^{\dbbar}\dbbar f\in \im \db$;
 i.e.\ $\omega \in \HB^{0,1}$.
 The conclusion now follows from \cref{mH-to-H-surjective}.
\end{proof}

Equation~\eqref{eqn:kL00-from-Q-form} implies that if $\kQ_{0,1}\geq0$, then the scalar GJMS operator is nonnegative.
Thus there are examples~\cite{Takeuchi2017}*{Theorem~1.6} of closed, strictly pseudoconvex five-manifolds for which $\kQ_{0,1}$ is not nonnegative.

We conclude this section by proving \cref{q01-nonnegative}.

\begin{proof}[Proof of \cref{q01-nonnegative}]
 The first conclusion is given by \cref{Q-nonnegative-consequence}.
 
 Suppose that $c_1(T^{1,0})=0$.
 Then~\cite{Case2021rumin}*{Theorem~7.5 and Lemma~19.7}
 \begin{equation*}
  [\ell^\theta] \in \im \left( -\Imaginary\db \colon H^{0,1}(M) \to H^{1}(M;\sP) \right) .
 \end{equation*}
 We deduce from \cref{Q-nonnegative-consequence} that $[\ell^\theta] \in \im \bigl( \HQ^{0,1} \to H^{0,1}(M) \to H^1(M;\sP) \bigr)$.
 The final conclusion now follows from \cref{scalar-q-flat-characterization,hodge-diagram-commutes-pluriharmonic}.
\end{proof}

\bibliography{bib}
\end{document}